\def\draftdate{November 2, 2010}

\documentclass{amsart}
\usepackage{xy}

\newdimen\wordsquish\wordsquish=1ex
\newdimen\redsquish\redsquish=0ex

\newcommand{\bs}{\backslash}

\newcommand{\ssdot}{\bullet}
\newcommand{\subdot}{_\ssdot}
\newcommand{\dsubdot}{_{\ssdot\ssdot}}

\newcommand{\Sdot}[1][\ssdot]{S_{#1}}
\newcommand{\Spdot}[1][\ssdot]{S'_{#1}}
\newcommand{\Spdotq}[1][\ssdot,\dotsc,\ssdot]{S^{\prime(q)}_{#1}}
\newcommand{\Spdotmac}[2]{S^{\prime(#1)}_{#2}}
\newcommand{\Sdotq}[1][\ssdot,\dotsc,\ssdot]{S^{(q)}_{#1}}

\newcommand{\co}{\mathrm{co}}
\newcommand{\w}{\mathrm{w}}

\newcommand{\COF}[1]{\aC\bs_{\co}#1}

\newcommand{\Lw}[1]{L#1_{\w}}
\newcommand{\Lco}[1]{L#1_{\co}}
\newcommand{\Laut}[1]{\HAut #1}\DeclareMathOperator\HAut{hAut}
\newcommand{\wco}{\widetilde\co}

\newcommand{\W}{\mathbf{W}}
\newcommand{\Wi}{\mathbf{W^{-1}}}
\newcommand{\C}{\mathbf{C}}
\newcommand{\WC}{\Wi\C}
\newcommand{\WCA}{\Wi\C_{A}}
\newcommand{\WW}{\Wi\W}
\newcommand{\WCW}{\Wi\C\Wi}
\newcommand{\WCdotsW}{\mathbf{W^{-1}C\dotsb W^{-1}}}

\mathchardef\varDelta="7101
\newcommand{\DDelta}{{\mathbf \varDelta}}
\mathchardef\varUpsilon="7107
\newcommand{\Word}{\mathbf{\varUpsilon}}
\mathchardef\varPsi="7109
\mathchardef\varXi="7104

\let\iso\cong
\newcommand{\htp}{\simeq}
\renewcommand{\to}{\mathchoice{\longrightarrow}{\rightarrow}{\rightarrow}{\rightarrow}}

\let\catsymbfont\mathcal
\newcommand{\aA}{{\catsymbfont{A}}}
\newcommand{\aB}{{\catsymbfont{B}}}
\newcommand{\aC}{{\catsymbfont{C}}}
\newcommand{\aD}{{\catsymbfont{D}}}

\newcommand{\aP}{{\catsymbfont{P}}}
\newcommand{\aS}{{\catsymbfont{S}}}

\def\quickop#1{\expandafter\DeclareMathOperator\csname
#1\endcsname{#1}}
\quickop{id}\quickop{Id}\quickop{colim}\quickop{hocolim}\quickop{op}
\quickop{Ar}\quickop{sing}\quickop{Hom}\quickop{Ho}
\quickop{hocoend}\quickop{Tel}

\quickop{MC}\quickop{Fac}\quickop{Map}
\newcommand{\MCp}{\MC^{s}}

\newcommand{\wcAr}{\Ar_{\wco}}

\DeclareMathOperator*\hocoendlim{\hocoend\mathstrut}

\numberwithin{equation}{section}

\newtheorem{thm}[equation]{Theorem}
\newtheorem*{thm*}{Theorem}
\newtheorem{cor}[equation]{Corollary}
\newtheorem{lem}[equation]{Lemma}
\newtheorem{prop}[equation]{Proposition}
\theoremstyle{definition}
\newtheorem{defn}[equation]{Definition}
\newtheorem{ex}[equation]{Example}
\newtheorem*{conv}{Convention}
\theoremstyle{remark}
\newtheorem{rem}[equation]{Remark}

\input xy
\xyoption{arrow}
\xyoption{curve}
\xyoption{matrix}
\xyoption{cmtip}
\SelectTips{cm}{}

\newdir{ >}{{}*!/-5pt/\dir{>}}

\newcommand{\term}[1]{\textit{#1}}

\begin{document}

\title%
{Algebraic $K$-theory and abstract homotopy theory}

\author{Andrew J. Blumberg}
\address{
Department of Mathematics, The University of Texas,
Austin, TX \ 78712}
\email{blumberg@math.utexas.edu}
\thanks{The first author was supported in part by an NSF postdoctoral
fellowship, NSF grant DMS-0111298, and a Clay Mathematics Institute Liftoff Fellowship}

\author{Michael A. Mandell}
\address{Department of Mathematics, Indiana University,
Bloomington, IN \ 47405}
\email{mmandell@indiana.edu}
\thanks{The second author was supported in part by NSF grants
DMS-0504069 and DMS-0804272.}

\date{\draftdate}

\begin{abstract}
We decompose the $K$-theory space of a Waldhausen category in terms of
its Dwyer-Kan simplicial localization.  This leads to a criterion for
functors to induce equivalences of $K$-theory spectra that generalizes
and explains many of the criteria appearing in the literature.
We show that under mild hypotheses,  a weakly exact
functor that induces an 
equivalence of homotopy categories induces an equivalence of
$K$-theory spectra. 
\end{abstract}

\maketitle

\section{Introduction}

Quillen's higher algebraic $K$-theory provides a powerful and subtle
invariant of rings and schemes.  Waldhausen reformulated the
definition and generalized the original input from algebra to
homological algebra or homotopy theory; in place of exact categories,
which are additive categories with a notion of exact sequence,
Waldhausen's construction allows categories equipped with weak
equivalences (quasi-isomorphisms) and a notion of cofibration
sequence.  Although designed to apply homotopy theory and $K$-theory
to geometric topology, the added flexibility of Waldhausen $K$-theory
turns out to be tremendously useful even when studying the original
algebraic objects.  For instance, the remarkable localization and
Mayer-Vietoris theorems of Thomason and Trobaugh
\cite[7.4,8.1]{TTGrothFest}, which relate the $K$-theory of a scheme to
the $K$-theories of open covers, depend on techniques possible
only in Waldhausen's framework.  One of the most important of these
techniques is the ability to change homological models, using
different categories of complexes with equivalent $K$-theory.  A
central question then becomes when do different models
yield the same $K$-theory \cite[1.9.9]{TTGrothFest}?  More generally,
what is $K$-theory made of?

In terms of comparing different models, Waldhausen's approximation
theorem \cite[1.6.4]{Wald} stands as the 
prototypical example of a $K$-theory equivalence criterion.  Thomason
and Trobaugh \cite[1.9.8]{TTGrothFest} specialized Waldhausen's
approximation theorem to certain categories of complexes, where for
appropriate complicial functors, an equivalence of derived categories
implies an equivalence of $K$-theory.  Based on this result and work
of the Grothendieck school on $K_0$, they articulated the perspective
that higher algebraic $K$-theory ``essentially depends only on the
derived category'' \cite[1.9.9]{TTGrothFest}, with a caveat about
choice of models.  Indeed, Schlichting \cite{SchlichtingDerivedIneq}
subsequently constructed examples of Frobenius categories with
abstractly equivalent derived categories but different algebraic
$K$-theory groups.  On the other hand, for the algebraic $K$-theory of
rings, Dugger and Shipley \cite{DuggerShipley} proved that an abstract
equivalence of derived categories does imply a $K$-theory
equivalence. 
Their argument relies on the folk theorem that a Quillen equivalence
of model categories induces an equivalence of $K$-theory of
appropriate Waldhausen subcategories.
To\"en and Vezzosi \cite{ToenVezzosi} generalized this from
Quillen equivalences to equivalences on
Dwyer-Kan simplicial localizations \cite{DKHammock}.  Other approaches
have tried to construct higher algebraic $K$-theory directly from the
derived category \cite{Neeman1} (and sequels) or using the
Heller-Grothendieck-Keller theory of ``derivators'' \cite{Grothderiv,
Grothpursue,Hellerhom,Kellerderiv} in for example \cite{MaltK}.

In this paper, we describe a precise relationship between the
algebraic $K$-theory space and the Dwyer-Kan simplicial localization
of the Waldhausen category.  To any category with weak equivalences,
the Dwyer-Kan simplicial localization associates simplicial mapping
spaces that have the ``correct'' homotopy type \cite{DKModel} and
that characterize the higher homotopy theory of the category
\cite[5.7]{MandellHH}.  Our description elucidates the nature of the
homotopical information encoded by $K$-theory, and leads to a very
general criterion for functors to induce an equivalence of $K$-theory
spectra, one that includes the approximation theorems above as special
cases.  We regard this decomposition as providing a conceptual
explanation of the phenomena described in the preceding paragraphs.

\begin{thm}\label{thmhocoendone}
Let $\aC$ be a Waldhausen category such that every map admits a 
factorization as a cofibration followed by a weak equivalence and
assume that the weak equivalences satisfy the
two out of three property.
For $n>1$, the nerve of $\w\Sdot[n]\aC$ is weakly equivalent to the homotopy coend
\[
\hocoendlim_{(X_{1},\dotsc,X_{n})\in \w\aC^{n}}L\aC(X_{n-1},X_{n})\times \dotsb \times
L\aC(X_{1},X_{2}),
\]
where $L\aC$ denotes the Dwyer-Kan hammock localization.
\end{thm}

The homotopy coend comes with a map to the classifying space
$B\w\aC^{n}$.  We can identify this space and the homotopy fiber of the
map intrinsically in terms of the Dwyer-Kan simplicial localization.
For $X$ an object in $\aC$, let $\Laut{X}$ denote the subspace of
$L\aC(X,X)$ consisting of the components corresponding to the weak
equivalences; precisely, $\Laut{X}$ consists of those components which
have a vertex where all the forward 
maps are weak equivalences. Then $\Laut(X)$ is a grouplike monoid of
homotopy automorphisms of $X$ in $L\aC$.  

\begin{thm}\label{thmhocoendtwo}
Let $\aC$ be as in Theorem~\ref{thmhocoendone}.  For
$n\geq 1$, the nerve of $\w\Sdot[n]\aC$ is weakly equivalent 
to the total space
of a fibration where the base is 
the disjoint union of 
\[
B\Laut{X_{n}}\times \dotsb \times B\Laut{X_{1}}
\]
over $n$-tuples of weak equivalences
classes of objects of $\aC$, and the fiber is equivalent to
\[
L\aC(X_{n-1},X_{n})\times \dotsb \times
L\aC(X_{1},X_{2})
\]
for $n>1$ and contractible for $n=1$.
\end{thm}

This description of the $K$-theory spaces may provide a
replacement in the abstract setting for certain $K$-theory arguments
that rely on the plus construction description, which is only
available for the $K$-theory of rings or connective ring spectra.
We expect this to apply to the study of Waldhausen's chromatic
convergence conjecture and related localization conjectures of
Rognes.  This is work in progress.

Currently, we can apply these theorems to the models question of Thomason
and Tro\-baugh.  We show that under mild hypotheses, a weakly 
exact functor that induces an equivalence of homotopy categories
induces an equivalence of $K$-theory spectra. The hypotheses hold in
particular in Waldhausen categories that come from model categories.
The main hypothesis is that any map in $\aC$ admits a
factorization as a cofibration followed by a weak 
equivalence; under this hypothesis, we say that $\aC$ 
\term{admits factorization}. 
Factorization generalizes Waldhausen's notion of ``cylinder
functor satisfying the cylinder axiom''.  
The secondary hypothesis involves the relationship between the weak
equivalences in the Waldhausen categories being compared.  One version
is the requirement (that often holds in practice) that the Waldhausen
categories  
have their weak equivalences closed under retracts; we have
included two alternative hypotheses for cases when this does not hold.
We prove the following theorem in Section~\ref{secpfmain}.  This
theorem can also be found in work of Cisinski
\cite{CisinskiUnPub}, where it is proved by other techniques.

\begin{thm}\label{main}
Let $\aC$ and $\aD$ be saturated Waldhausen categories that admit
factorization.  Let $F\colon \aC\to \aD$ be a 
weakly exact functor 
that induces an equivalence on homotopy categories.  If one of the
following additional hypotheses holds
\begin{enumerate}
\item The weak equivalences of $\aC$ and $\aD$ are closed under retracts,
\item A map $f$ in $\aC$ is a weak equivalence if and
only the map $Ff$ in $\aD$ is a weak equivalence, or
\item For any $A,B\in \aC$, the image of $\Ho(\w\aC)(A,B)$ in
$\Ho\aD(FA,FB)$ coincides with the image of $\Ho(\w\aD)(FA,FB)$, 
\end{enumerate}
then $F$ induces
an equivalence of $K$-theory spectra.
\end{thm}

In the statement, a ``weakly exact'' functor is a homotopical
generalization of an exact functor.  An exact functor between
Waldhausen categories preserves weak equivalences,
cofibrations, and pushouts along cofibrations.  A weakly exact functor
preserves weak equivalences, but need only preserve
cofibrations and pushouts along cofibrations up to weak equivalence 
(see Definition~\ref{defweak} below).  The ``homotopy category'' of a
category $\aC$ with weak equivalence is the category $\Ho\aC$ obtained by
formally inverting the 
weak equivalences. $\Ho\aC$ generalizes the derived category to
this context; it is typically not a triangulated category without
additional hypotheses on $\aC$.

Following Waldhausen's notation, we have used $\w\aC$ and $\w\aD$ to
denote the subcategories of weak equivalences for $\aC$ and $\aD$.
The image of $\Ho(\w\aC)$ in $\Ho\aC$ consists of isomorphisms (by
definition), but might not in general contain all isomorphisms of
$\Ho\aC$.  It does contain all the isomorphisms, however, under the
hypotheses of Theorem~\ref{main} when
the weak equivalences of $\aC$ are
closed under retracts; see Section~\ref{sechococart} for a complete
discussion.  Hypotheses~(ii) and~(iii) in Theorem~\ref{main} ensure that the weak
equivalences of $\aC$ and $\aD$ and their formal inverses generate
equivalent subcategories of $\Ho\aD$ even when they do not necessarily
generate all the isomorphisms of $\Ho\aD$.

In the proof of 
Theorem~\ref{main}, we argue that a weakly exact functor that induces
an equivalence of homotopy categories comes very near to being a
DK-equivalence (a functor that induces a weak equivalence of Dwyer-Kan
simplicial localizations); see Corollary~\ref{corsuspdkequiv}.  
It remains an interesting
question to determine when such a functor is a DK-equivalence.  When we drop the
weakly exact hypothesis and consider only functors that preserve weak
equivalences, we can characterize 
DK-equivalences in terms of homotopy
categories of undercategories.  For an object $A$ of $\aC$, let
$\aC\bs A$ denote the category of objects in $\aC$ under $A$, i.e.,
an object consists of a map $A\to X$ in $\aC$ and a map from $A\to X$
to $A\to Y$ consists of a map $X\to Y$ in $\aC$ that commutes with the
maps from $A$; say that such a map is a weak equivalence when its
underlying map $X\to Y$ is a weak equivalence in $\aC$.  We can then
form the homotopy category $\Ho(\aC \bs A)$ by formally inverting the
weak equivalences.  We prove in Section~\ref{secphh} the following
theorem generalizing the main result of \cite{MandellHH}.

\begin{thm}\label{thmhh}
Let $\aC$ and $\aD$ be saturated Waldhausen categories that admit
factorization, and let $F\colon \aC\to \aD$ be a functor that
preserves weak equivalences.  Then $F$ is a DK-equivalence if and only
if it induces an equivalence $\Ho(\aC)\to \Ho(\aD)$ and an equivalence
$\Ho(\aC \bs A)\to \Ho(\aD \bs FA)$ for all objects $A$ of $\aC$.
\end{thm}

This interpretation relates to Waldhausen's approximation
theorem and provides a conceptual understanding of the role of
Waldhausen's approximation property \cite[1.6.4]{Wald} in the more
specialized approximation theorems.  Recall that
for Waldhausen categories $\aC$ and $\aD$, an exact functor $F\colon
\aC \to \aD$ satisfies the approximation property if
\begin{enumerate}
\item A map $f\colon A \to B$ is a weak equivalence in $\aC$ if and
only if the map $F(f)\colon FA \to FB$ is a weak equivalence
in $\aD$.

\item For every map $FA \to X$ in $\aD$, there exists a
cofibration $A \to B$ in $\aC$ and a weak equivalence
$FB \to X$ in $\aD$ such that the diagram
\[
\xymatrix{
FA \ar[r] \ar@{ >->}[d] & X \\
FB \ar[ur]_{\htp} \\
}
\]
commutes.
\end{enumerate}
We prove the following theorem in Section~\ref{secpwaldapp}.

\begin{thm}\label{propapprox}
Let $\aC$ be a saturated Waldhausen category that admits
factorization.   Let $\aD$ be a saturated
Waldhausen category, and let $F\colon 
\aC\to \aD$ be an exact functor. If $F$ satisfies
Waldhausen's approximation property, then $\Ho(\aC)\to \Ho(\aD)$ is an
equivalence and $\Ho(\aC \bs A)\to \Ho(\aD \bs FA)$ is an equivalence
for all objects $A$ of $\aC$.
\end{thm}

In all of the preceding theorems, we required the
hypothesis that the Waldhausen categories be saturated, meaning that the
weak equivalences satisfy the two out of three property: For
composable maps $f$ and $g$, if any two 
of $f$, $g$, and $f\circ g$ are weak equivalences then so is the third.
This usage of the term ``saturated'' differs from the usage of the term in other
sources such as \cite{DKHS}.  Although much of the work
in this paper could be adjusted to avoid this hypothesis, it is a
hypothesis so common and pervasive in homotopy theory that to
do so would lose more in the awkwardness it would engender than it
would gain in the extra abstract generality it might achieve. Rather
than continually repeating this hypothesis throughout the rest of the
paper, we instead incorporate it 
by convention in the definition of weak equivalences.

\begin{conv}
In this paper we always understand a subcategory of weak equivalences to
satisfy the two out of three property.  In particular, all Waldhausen
categories are assumed to be saturated in the sense of Waldhausen.
\end{conv}

Finally, we should note that in virtually every example of interest,
the factorizations hypothesized in the theorems above tend to be
functorial.  Assuming functorial factorizations simplifies many of the
arguments; for these arguments, we assume functorial factorization in
the body of the paper and treat the non-functorial case in the appendices.

\bigskip

The authors would like to thank the Institute for Advanced Study and
the University of Chicago for their hospitality while some of this
work was being done. The authors would also like to thank H. Miller
and L. Hesselholt for asking motivating questions.

\section{Weakly exact functors}\label{secweak}

This section defines weakly exact functor and shows that under mild
technical hypotheses a weakly exact functor between Waldhausen
categories induces a map between their $K$-theory spectra.  Although
we expect that the extra flexibility provided by stating
Theorem~\ref{main} in terms of weakly exact functors
rather than exact functors will increase its applicability, in fact,
weakly exact functors play a vital 
technical role in its proof even in the case when the functor in
question is exact.  Specifically, the proof requires a version of 
Theorem~\ref{thmhocoendone} that is natural in weakly exact
functors, which we state as Theorem~\ref{thmhocoend} at the end of the
section.  We begin with the definition of weakly exact
functor. 

\begin{defn}\label{defweak}
Let $\aC$ and $\aD$ be Waldhausen categories.  A functor $F\colon
\aC\to \aD$ is \term{weakly exact} if the initial map $*\to F*$ in
$\aD$ is a weak equivalence and $F$ preserves weak equivalences, weak
cofibrations, and homotopy cocartesian squares.
\end{defn}

In the definition, a \term{weak cofibration} is a map that is weakly
equivalent (by a zigzag) to a cofibration in the category $\Ar\aC$ of
arrows in $\aC$, and a \term{homotopy 
cocartesian square} is a
square diagram that is weakly equivalent (by a
zigzag) to a pushout square where one of the parallel sets of arrows
consists of cofibrations.  It follows that a functor that preserves
weak equivalences will preserve weak cofibrations and homotopy
cocartesian squares if and only if it takes cofibrations to weak
cofibrations and takes pushouts along cofibrations to homotopy
cocartesian squares.
Clearly the concept of weakly exact functor
will only be useful when homotopy cocartesian squares have the usual
expected properties.   According to \cite[\S2]{BlumbergMandell}, these
properties hold when the Waldhausen category admits ``functorial
factorization of weak cofibrations''.  (See Appendix~\ref{appbm}
for a non-functorial generalization.)

Recall that a Waldhausen category $\aC$ admits \term{functorial factorization}
when any map $f\colon A\to B$ in $\aC$ factors as a cofibration
followed by a weak equivalence 
\[
\xymatrix{%
A\ar@{ >->}[r]\ar@{..>}@/_1em/[rr]_{f}&Tf\ar[r]^{\htp}&B,
}
\]
functorially in $f$ in the category $\Ar\aC$ of arrows in $\aC$.  In
other words, given the map $\phi$ of arrows on the 
left (i.e., commuting diagram),
\[
\xymatrix{%
A\ar[r]^{f}\ar[d]_{a}\ar@{{}{}{}}[dr]|{\phi}&B\ar[d]^{b}
&&A\ar@{ >->}[r]\ar[d]_{a}&Tf\ar[r]^{\htp}\ar[d]_{T\phi}&B\ar[d]^{b}\\
A'\ar[r]_{f'}&B'&&A'\ar@{ >->}[r]&Tf'\ar[r]^{\htp}&B'\\
}
\]
we have a map $T\phi$ that makes the diagram on the right commute and
that satisfies the usual identity and composition relations,
$T\id_{f}=\id_{Tf}$ and $T(\phi'\circ \phi)=T\phi'\circ T\phi$.
A
cylinder functor satisfying the cylinder
axiom in the sense of Waldhausen \cite[\S1.6]{Wald} is a
factorization functor
that in addition satisfies strong exactness properties. 

In Waldhausen categories that admit functorial factorization, every
map is weakly equivalent to a 
cofibration.  This isn't always the case in examples of interest,
especially in ``Waldhausen subcategories''.  To
get around this, in \cite{BlumbergMandell} we worked in terms of 
the technical hypothesis that $\aC$ admit \term{functorial
factorization of weak cofibrations (FFWC)}
\cite[2.2]{BlumbergMandell}, which means that the weak cofibrations
can be factored functorially (in $\Ar\aC$) as above.  
Our interest in FFWC is the following theorem proved in this section.

\begin{thm}\label{thmweakexact}
Let $F\colon \aC \to \aD$ be a weakly exact functor between Waldhausen
categories and assume that $\aD$ admits FFWC.  Then
$F$ induces a map of $K$-theory spectra.
\end{thm}

We prove Theorem~\ref{thmweakexact} using the $\Spdot$
construction of \cite[\S2]{BlumbergMandell}.  To put this in context,
we begin by reviewing the 
$\Sdot$ construction in detail.  Recall that Waldhausen's $\Sdot$
construction produces a simplicial Waldhausen category $\Sdot\aC$ from
a Waldhausen category $\aC$ and is defined as follows.  Let $\Ar[n]$
denote the category with objects $(i,j)$ for $0\leq i\leq j\leq n$ and
a unique map $(i,j)\to (i',j')$ for $i\leq i'$ and $j\leq j'$.
$\Sdot[n]\aC$ is defined to be the full subcategory of the category of
functors $A\colon \Ar[n]\to \aC$ such that:
\begin{itemize}
\item $A_{i,i}=*$ for all $i$, 
\item The map $A_{i,j}\to A_{i,k}$ is a cofibration for all $i \leq
j \leq k$, and
\item The diagram
\[  \xymatrix@-1pc{%
A_{i,j}\ar[r]\ar[d]&A_{i,k}\ar[d]\\A_{j,j}\ar[r]&A_{j,k}
} \]
is a pushout square for all $i \leq j \leq k$, 
\end{itemize}
where we write $A_{i,j}$ for $A(i,j)$.  The last two conditions can be
simplified to the hypothesis that each map $A_{0,j}\to A_{0,j+1}$ is a
cofibration and the induced maps $A_{0,j}/A_{0,i}\to A_{i,j}$ are
isomorphisms.  This becomes a Waldhausen category by defining a map
$A\to B$ to be a weak equivalence when each $A_{i,j}\to B_{i,j}$ is a
weak equivalence in $\aC$, and to be a cofibration when each
$A_{i,j}\to B_{i,j}$ and each induced map
$A_{i,k}\cup_{A_{i,j}}B_{i,j}\to B_{i,k}$ is a cofibration in $\aC$.
The following definition gives a homotopical version of this
construction for Waldhausen 
categories that admit FFWC.

\begin{defn}
Let $\aC$ be a Waldhausen category that admits FFWC.  Define
$\Spdot[n]\aC$ to be the full subcategory of functors $A\colon
\Ar[n]\to \aC$ such that: 
\begin{itemize}
\item The initial map $*\to A_{i,i}$ is a weak equivalence for all $i$, 
\item The map $A_{i,j}\to A_{i,k}$ is a weak cofibration for all $i
\leq j \leq k$, and 
\item The diagram
\[  \xymatrix@-1pc{%
A_{i,j}\ar[r]\ar[d]&A_{i,k}\ar[d]\\A_{j,j}\ar[r]&A_{j,k}
} \]
is a homotopy cocartesian square for all $i \leq j \leq k$.
\end{itemize}
We define a map $A\to B$ to be a
weak equivalence when each $A_{i,j}\to B_{i,j}$ is a weak equivalence
in $\aC$, and to be a cofibration when each $A_{i,j}\to B_{i,j}$ is
a cofibration in $\aC$ and each induced map
$A_{i,k}\cup_{A_{i,j}}B_{i,j}\to B_{i,k}$ is a weak cofibration in
$\aC$.
\end{defn}

Clearly $\Spdot \aC$ assembles into a simplicial category with the usual
face and degeneracy functors.  Furthermore, we have the following
comparison result \cite[2.8,2.9]{BlumbergMandell}.

\begin{prop}\label{propbm}
Let $\aC$ be a Waldhausen category admitting FFWC.  
\begin{enumerate}
\item $\Spdot\aC$ is a simplicial Waldhausen category admitting FFWC.
\item The inclusion $\Sdot \aC \to \Spdot \aC$ is a simplicial exact
functor.
\item For each $n$, the inclusion $\w\Sdot[n]\aC\to \w\Spdot[n]\aC$ induces a weak
equivalence on nerves.
\end{enumerate}
\end{prop}

The following proposition is now clear from the definition of weakly
exact.  Theorem~\ref{thmweakexact} then follows from this
proposition and the previous proposition by
iterating the $\Spdot$ construction.

\begin{prop}\label{propwexact}
Let $F\colon \aC \to \aD$  a weakly exact functor between Waldhausen
categories and assume that $\aD$ admits FFWC.  Then for each $n$,  $F$
sends $\Sdot[n]\aC$ into $\Spdot[n]\aD$ by a weakly exact functor. 
\end{prop}

Finally, we can use the $\Spdot$ construction to express the full
naturality of the weak equivalences in Theorem~\ref{thmhocoendone}.
Unfortunately, the hypothesis FFWC is not
quite strong enough.  We need a slight refinement of this hypothesis:

\begin{defn}\label{defFMCWC}
Let $\aC$ be a Waldhausen category.  We say that $\aC$ has \term{functorial
mapping cylinders for weak cofibrations (FMCWC)} when $\aC$ admits
functorial factorization of weak cofibrations by a functor $T$
together with a natural transformation $B \to Tf$ splitting the
natural weak equivalence $Tf\to B$, for weak cofibrations $f\colon
A\to B$.
\[
\xymatrix@R-1pc{%
&B\ar[d]\ar[dr]^{=}\\
A\ar@{ >->}[r]\ar@{..>}@/_1em/[rr]_{f}&Tf\ar[r]^{\htp}&B,
}
\]
\end{defn}

Functorial factorization of all maps implies functorial
mapping cylinders: For a map $f\colon A\to B$, the factorization of
the map $f+\id_{B}$ from the coproduct, $A\vee B\to B$, provides the functorial
mapping cylinder.  Thus, functorial factorization of all maps is
equivalent to the conjunction of functorial mapping cylinders for weak
cofibrations and all maps being weak cofibrations.  

For a Waldhausen category $\aC$ that has functorial mapping cylinders
for weak cofibrations, and $A,B$ objects in $\aC$, we use
$\Lco\aC(A,B)$ to denote the components of the Dwyer-Kan hammock
function complex $L\aC(A,B)$ that correspond to the weak cofibrations;
precisely, $\Lco\aC(A,B)$ consists of those components that contain as
a vertex a zigzag where all the forward arrows are weak cofibrations.
Likewise, we use $\Lw\aC(A,B)$ to denote the components of the
Dwyer-Kan hammock function complex $L\aC(A,B)$ that contain a zigzag
where all the forward arrows are weak equivalences.  Then $\Lco\aC$
and $\Lw\aC$ are simplicial subcategories of the Dwyer-Kan simplicial
localization $L\aC$.  We prove the following generalization of
Theorems~\ref{thmhocoendone} and~\ref{thmhocoendtwo} in
Section~\ref{secpdecomp}.  (See Appendix~\ref{appmain} for the
corresponding non-functorial statement.)

\begin{thm}\label{thmhocoend}
Let $\aC$ be a saturated Waldhausen category that has FMCWC. 
\begin{enumerate}
\item
For $n>1$, the nerve of $\w\Spdot[n] \aC$ is weakly equivalent to
the homotopy coend 
\[
\hocoendlim_{(X_{1},\dotsc,X_{n})\in \w\aC^{n}}\Lco\aC(X_{n-1},X_{n})\times \dotsb \times
\Lco\aC(X_{1},X_{2}),
\]
naturally in weakly exact functors.  
\item The nerve of $\w\aC$ is weakly equivalent to the disjoint
union of $B\Laut{X}$ over the weak equivalence classes of objects of
$\aC$.
\item
For $n\geq 1$, the nerve of $\w\Spdot[n] \aC$ is weakly equivalent
to the total space
of a fibration where the base is 
the disjoint union of 
\[
B\Laut{X_{n}}\times \dotsb \times B\Laut{X_{1}}
\]
over $n$-tuples of weak equivalences
classes of objects of $\aC$, and the fiber is equivalent to
\[
\Lco\aC(X_{n-1},X_{n})\times \dotsb \times
\Lco\aC(X_{1},X_{2}).
\]
for $n>1$ and contractible for $n=1$.
\end{enumerate}
\end{thm}


\section{Outline of the proof of Theorem~\ref{main}}\label{secpfmain}

In this section we prove Theorem~\ref{main} from
Theorem~\ref{thmhocoend} above, Theorem~\ref{thmprehococart}
below, and Propositions~\ref{prophypi}  and~\ref{prophypii} below,
all of which are proved in later sections.   Throughout this section
(and this section only), we fix $\aC$, $\aD$, and $F\colon \aC\to \aD$,
satisfying the hypotheses of Theorem~\ref{main} and such that
factorization is functorial. (See Appendix~\ref{appmain} for the proof
in the non-functorial case.) Moreover, we fix factorization functors
$T$ on $\aC$ and $\aD$; we use the following terminology and notation.

\begin{defn}
For $X$ in $\aC$ (resp. $\aD$), the \term{cone on $X$}, $CX$, is $T(X\to *)$ and the
\term{suspension of $X$}, $\Sigma X$, is $CX/X$.  Let $EX$ denote the cofiber
sequence
\[
\xymatrix@-1pc{%
X\ar@{ >->}[r]&CX\ar[r]&\Sigma X
}
\]
viewed as an object of $S_{2}\aC$ (resp. $S_{2}\aD$), with $A_{0,1}=X$,
$A_{0,2}=CX$, and $A_{1,2}=\Sigma X$.
\end{defn}

It follows from the functoriality of $T$ that $CX$, $\Sigma X$, and
$EX$ assemble to functors in $X$.  A straightforward application of
factorization and \cite[2.5]{BlumbergMandell} (or the gluing axiom) shows
that these functors preserve 
weak equivalences and homotopy cocartesian squares.  This gives the
following proposition; the corresponding result holds for $\aD$.

\begin{prop}
The functors $C$ and $\Sigma$ are weakly exact functors $\aC\to \aC$,
and $E$ is a weakly exact functor $\aC\to S_{2}\aC$.
\end{prop}

The factorization functor for $\aC$ induces a factorization functor on
$S_{2}\aC$, and so $E$ induces a map of $K$-theory spectra $K\aC\to
KS_{2}\aC$.  Applying the Additivity Theorem
\cite[1.4.2,1.3.2.(3)]{Wald}, we see that on $K$-theory, the sum
in the stable category of the maps induced by the identity and
suspension is the map induced by the cone.  Since the cone induces the
trivial map, it follows that $\Sigma$ induces on
$K\aC$ the map $-1$ in the stable category.  In particular, we obtain
the following corollary; the corresponding result holds for $\aD$.

\begin{cor}\label{corsuspequiv}
$\Sigma$ induces a weak
equivalence on $K$-theory spectra $K\aC\to K\aC$.
\end{cor}

Although we do not assume any
relationship between the factorization functors on $\aC$
and $\aD$, nevertheless, we can relate the suspensions.

\begin{prop}\label{propsusphty}
There is a functor $\Xi\colon \aC\to \aD$ and natural weak equivalences
\[
\xymatrix@-.5pc{%
F\Sigma X&\Xi X\ar[l]_{\htp}\ar[r]^{\htp}&\Sigma FX.
}
\]
\end{prop}

For example, $\Xi X$ can be defined as the pushout
\[
\xymatrix@C-1pc{%
FX\ar@{ >->}[r]\ar[d]&T(F(X\to CX))\ar[d]\\
F*\ar[r]&\Xi X.
}
\]
The factorization weak equivalence $T(F(X\to CX))\to FCX$ and the
universal property of the pushout induces the map $\Xi X\to F\Sigma X$,
which is a weak equivalence \cite[2.5]{BlumbergMandell}.
Functoriality of $T$
in $\Ar\aD$ then gives a map under $FX$,  
\[
T(F(X\rightarrow CX))\to T(FX \rightarrow *) = CFX,
\]
which is a weak equivalence since the initial map to each is a weak
equivalence.  This map and the final map $F*\to *$ induce the map
$\Xi X\to \Sigma FX$, which is a 
weak equivalence by the gluing axiom. 

To take advantage of the suspension functor, we need to relate it to
the Dwyer-Kan function complexes.  For this we use the following
application of Theorem~\ref{thmhococart}  from
Section~\ref{sechococart}.  Again, the corresponding theorem
also holds for $\aD$. 

\begin{thm}\label{thmprehococart}
If the diagram on the left below is homotopy cocartesian in $\aC$,
\[
\xymatrix@-1pc{%
A\ar[r]\ar[d]&B\ar[d]
&&L\aC(D,Y)\ar[r]\ar[d]&L\aC(C,Y)\ar[d]\\
C\ar[r]&D
&&L\aC(B,Y)\ar[r]&L\aC(A,Y)
}
\]
then for any object $Y$ in $\aC$, the diagram on the right is homotopy
cartesian in the category of simplicial sets.
\end{thm}

Applying this theorem to 
the homotopy cocartesian square defining the suspension, we obtain the
following corollary. 

\begin{cor}
For any $X,Y$ in $\aC$, 
$L\aC(\Sigma X,Y)$ is weakly equivalent to the based loop space of
$L\aC(X,Y)$, based at the 
trivial map $X\to Y$.
\end{cor}

Applying $F$ to the homotopy cocartesian square defining the
suspension, we see that likewise $L\aD(F\Sigma X,FY)$ is weakly
equivalent to the based loop space of
$L\aD(FX,FY)$, based at $F$ of the trivial map $X\to Y$, or
equivalently, based at the trivial map $FX\to FY$ since it
is in the same component.

Iterating the suspension in the previous proposition, we see that
$\pi_{n}L\aC(X,Y)$ based at the trivial map is 
\[
\pi_{0}L\aC(\Sigma^{n}X,Y) = \Ho\aC(\Sigma^{n}X,Y).
\]
Since $F$ induces an equivalence $\Ho\aC\to \Ho\aD$, we obtain the
following corollary. 

\begin{cor}\label{corsuspdkequiv}
For any $X,Y$ in $\aC$, the restriction of 
\[
LF\colon L\aC(X,Y)\to L\aD(FX,FY)
\]
to the component of the trivial
map is a weak equivalence; the map
\[
LF\colon L\aC(\Sigma X,Y)\to L\aD(F\Sigma X,FY)
\]
is a weak equivalence. 
\end{cor}

To apply the previous corollary, we also need to know that the
components of $L\aC(\Sigma X,\Sigma X)$ in $\Laut{\Sigma X}$
correspond to the same components of $L\aD(F\Sigma X,F\Sigma X)$ that
are in $\Laut{F\Sigma X}$.  Since these components consist of exactly
the components representing the image 
of $\Ho\w\aC(\Sigma X,\Sigma X)$ and $\Ho\w\aD(F\Sigma X,F\Sigma X)$,
respectively, this is clear when hypothesis~(iii) in the statement of
Theorem~\ref{main} holds.  The other two cases are handled by the
following propositions.  The first is a special case of
Theorem~\ref{thmTFAE}.

\begin{prop}\label{prophypi}
If hypothesis~(i) in the statement of
Theorem~\ref{main} holds, then a map in $\aC$ is a weak equivalence if
and only if it represents an isomorphism in $\Ho\aC$, and likewise for $\aD$.
\end{prop}

The second proposition is proved in Section~\ref{sechclf} as 
Corollary~\ref{corhypii}.

\begin{prop}\label{prophypii}
If hypothesis~(ii) in the statement of Theorem~\ref{main} holds, then
so does hypothesis~(iii).
\end{prop}

We have now assembled all we need to prove Theorem~\ref{main}.  By
Proposition~\ref{propbm}, we can use $N(\w\Spdot\aC)$, the diagonal of
the nerve of the simplicial category $\w\Spdot\aC$, as a model for the
(delooped) $K$-theory space of $\aC$.
Since suspension induces a weak equivalence on $K$-theory, the
telescope under suspension
\[
\Tel_{\Sigma} N(\w\Spdot \aC)
= \Tel (
\xymatrix@-1pc{%
N(\w\Spdot\aC)\ar[r]^-{\Sigma}&
N(\w\Spdot\aC)\ar[r]^-{\Sigma}&
N(\w\Spdot\aC)\ar[r]^-{\Sigma}&\dotsb 
})
\]
is equivalent to $N(\w\Spdot\aC)$ via the inclusion.  The same
observations apply to $\aD$, and Proposition~\ref{propsusphty}
provides homotopies to construct a map of telescopes
\begin{equation}\label{eqtel}
\Tel_{\Sigma} N(\w\Spdot[n]\aC)\to 
\Tel_{\Sigma} N(\w\Spdot[n]\aD)
\end{equation}
for all $n$.  Now we use the models from Theorem~\ref{thmhocoend}.  We write
\[
L\aC(X_{1},\dotsc,X_{n})=L\aC(X_{n-1},X_{n})\times \dotsb \times L\aC(X_{1},X_{2})
\]
and similarly for $\aD$ to save space.  Then by
Theorem~\ref{thmhocoend}.(i), the square in the homotopy category
formed by the $\Spdot[n]$ constructions
\[
\xymatrix{%
N(\w\Spdot[n]\aC)\ar[r]^{F_{*}}\ar[d]_{\Sigma_{*}}
&N(\w\Spdot[n]\aD)\ar[d]^{\Sigma_{*}}\\
N(\w\Spdot[n]\aC)\ar[r]_{F_{*}}
&N(\w\Spdot[n]\aD)
}
\]
is isomorphic in the homotopy category to the square in the homotopy
category formed by the homotopy coends
\[
\xymatrix{%
\hocoend L\aC(X_{1},\dotsc,X_{n})\ar[r]^{LF_{*}}\ar[d]_{L\Sigma_{*}}
&\hocoend L\aD(Y_{1},\dotsc,Y_{n})\ar[d]^{L\Sigma_{*}}\\
\hocoend L\aC(X_{1},\dotsc,X_{n})\ar[r]_{LF_{*}}
&\hocoend L\aD(Y_{1},\dotsc,Y_{n}),
}
\]
where the homotopy coends are over $(X_{1},\dotsc,X_{n})\in \w\aC^{n}$
and $(Y_{1},\dotsc,Y_{n})\in \w\aD^{n}$.  Writing
$\aC_{\Sigma}$ and
$\aD_{\Sigma}$ for the full subcategories of $\aC$ and $\aD$
consisting of objects that are weakly equivalent to suspensions, it is
clear that the vertical map of arrows factors through the arrow
\[
\hocoendlim_{(X_{1},\dotsc,X_{n})\in \w\aC^{n}_{\Sigma }} 
L\aC(X_{1},\dotsc,X_{n})\quad 
\xrightarrow{\ LF_{*}\ }
\hocoendlim_{(Y_{1},\dotsc,Y_{n})\in \w\aD^{n}_{\Sigma }}
L\aD(Y_{1},\dotsc,Y_{n}).
\]
Corollary~\ref{corsuspdkequiv} and
Theorem~\ref{thmhocoend} imply this latter map is a weak
equivalence, and we conclude that the map of telescopes~\eqref{eqtel}
is a weak equivalence.  This completes the proof of Theorem~\ref{main}.

\section{Universal simplicial quasifibrations}

In this section, we introduce the first of two techniques which
provide the foundation for our subsequent work in this paper.  
The proofs of the theorems in the previous sections depend on
machinery for solving two related problems: The 
identification of certain squares of simplicial sets as homotopy
cartesian squares and the identification of the homotopy fiber of
certain maps of simplicial sets.  Quillen's Theorem~B \cite{QuillenAK}
and its simplicial variant \cite[1.4.B]{Wald} provide a flexible tool
for these purposes.  We rely on a particular formulation of Theorem~B
in terms of a notion of ``universal simplicial quasifibration''.  Our
exposition and viewpoint on the subject is heavily influenced by
postings of Tom Goodwillie on Don Davis' algebraic topology mailing
list \cite{Goodwillie-Davis}.

Recall that a map $X\to Y$ of spaces is a quasifibration when for
every point $x$ of $X$, the map from the fiber to the homotopy fiber
is a weak equivalence. We say that a map of simplicial sets $X\subdot
\to Y\subdot$ is a quasifibration when its geometric realization is a
quasifibration of spaces.

\begin{defn}\label{defunivsq}
A map of simplicial sets $X\subdot \to Y\subdot$ is a \term{universal simplicial
quasifibration} when for every map of simplicial sets $Z\subdot\to
Y\subdot$, the induced map of the pullback $X\subdot \times_{Y\subdot}
Z\subdot\to Z\subdot$ is a quasifibration.
\end{defn}

The definition specifies a class of maps for which it is easy to
identify pullbacks as homotopy pullbacks.  The following proposition
implies that to verify that a map is a universal simplicial
quasifibration, it suffices to check the condition on the simplexes of
$X\subdot$; for the simplices, checking that the pullback map to the simplex
is a quasifibration then 
amounts to checking that the fiber over a vertex includes as a weak
equivalence.  The proof in one direction is the restriction of the universal 
simplicial quasifibration property to the standard simplices; the
proof in the other direction is Waldhausen's version of Theorem~B
\cite[1.4.B]{Wald}.

\begin{prop}\label{propunivsq}
A map of simplicial sets $X\subdot \to Y\subdot$ is a universal simplicial
quasifibration if and only if for every $n$ and every map $\Delta[n]\to
Y\subdot$, the induced map of the pullback $X\subdot \times_{Y\subdot}
\Delta[n]\to \Delta[n]$ is a quasifibration.
\end{prop}

In general, the simplicial sets we use arise as homotopy colimits.
Thus, it is convenient to state the following proposition.  For a
small category $\aC$, we write $N\aC$ to denote the simplicial
nerve.

\begin{prop}\label{prophocolimb}
Let $\aC$ be a small category and $F$ a functor from $\aC$ to simplicial
sets.  Suppose that for every map $f\colon x\to y$ in $\aC$, the
induced map $F(f)\colon F(x)\to F(y)$ is a weak equivalence.  Then the
map $\hocolim_{\aC}F\to N\aC$ is a universal simplicial quasifibration.
\end{prop}

\begin{proof}
We can identify an $n$-simplex of $N\aC$ as a functor
$\sigma \colon \DDelta[n]\to \aC$, where $\DDelta[n]$ is the poset of $0,\dotsc,n$
under $\leq$.  The pullback of $\hocolim_{\aC}F$ over this simplex is
$\hocolim_{\DDelta[n]} F\circ \sigma$.  For any vertex $i$, the inclusion
of the fiber, $F(\sigma(i))$, in $\hocolim_{\DDelta[n]} F\circ \sigma$
is a weak equivalence.
\end{proof}

The same proof also gives the following proposition, which we apply
directly in the proof of Theorem~\ref{thmhocoend}.

\begin{prop}\label{prophocoend}
Let $\aC$ be a small category and $F$ a functor from $\aC^{\op}\times
\aC$ to simplicial 
sets.  Suppose that for every $z$ in $\aC$ and every map $f\colon x\to y$ in $\aC$, the
induced maps $F(f,z)\colon F(y,z)\to F(x,z)$ and $F(z,f)\colon
F(z,x)\to F(z,y)$ are weak equivalences.  Then the
map $\hocoend_{\aC}F\to N\aC$ is a universal simplicial quasifibration.
\end{prop}

We also repeatedly use the following refinement of Quillen's Theorem~B
\cite{QuillenAK}.  Recall that for a functor $\phi \colon\aD \to \aC$ and
a fixed object $Z$ of $\aC$, the comma category $Z \downarrow \phi$ has
as objects the pairs $(Y, Z \to \phi Y)$ consisting of an object $Y$ of $\aD$ and
a map in $\aC$ from $Z$ to $\phi Y$.  A morphism 
\[(Y, Z\rightarrow \phi Y) \to (Y', Z \rightarrow \phi Y')\]
consists of a map $Y \to Y'$ in $\aD$ such that the diagram
\[
\xymatrix@R-1pc@C-2pc{
&Z \ar[dr]\ar[dl]\\ 
\phi Y \ar[rr] 
&& \phi Y'
}
\]
in $\aD$ commutes.  The following theorem gives a useful sufficient condition for a
commuting square of functors to induce a homotopy cartesian square of
nerves.   

\begin{thm}\label{thmC}
Let $\aA$, $\aB$, $\aC$, $\aD$ be small categories, and let
\[
\xymatrix@-1pc{%
\aD\ar[r]^{\delta}\ar[d]&\aC\ar[d]^{\gamma}\\
\aB\ar[r]_{\beta}&\aA
}
\]
be a (strictly) commuting diagram of functors. If the following
two conditions hold, then the induced square of nerves is homotopy
cartesian: 
\begin{enumerate}
\item For every map $X\to X'$
in $\aA$, the induced functor on comma categories $X'\downarrow \beta\to
X\downarrow \beta$ induces a weak equivalence of nerves, and
\item For every object $Z$ in $\aC$, the functor $Z\downarrow\delta
\to \gamma Z\downarrow \beta$ induces a weak equivalence of nerves.
\end{enumerate}
\end{thm}

\begin{proof}
As in Quillen's proof of Theorem~B, we have natural weak equivalences
\[
\hocolim_{X\in\aA} N(X\downarrow \beta) \to N\aB\quad\text{and}\quad 
\hocolim_{Z\in\aC} N(Z\downarrow \delta) \to N\aD.
\]
These then fit into the commutative diagram on the left
\[
\xymatrix@-1pc{%
\hocolim_{Z\in\aC} N(Z\downarrow \delta)\ar[d]\ar[r]&N\aC\ar[d]
&&\hocolim_{Z\in\aC} N(\gamma Z\downarrow \beta)\ar[d]\ar[r]&N\aC\ar[d]\\
\hocolim_{X\in\aA} N(X\downarrow \beta)\ar[r]& N\aA
&&\hocolim_{X\in\aA} N(X\downarrow \beta)\ar[r]& N\aA
}
\]
weakly equivalent to the diagram of nerves in question.  By (ii), we
have that the canonical map
\[
\hocolim_{Z\in\aC} N(Z\downarrow \delta)\to \hocolim_{Z\in\aC} N(\gamma Z\downarrow \beta)
\]
is a weak equivalence, and it follows that the square on the right
above is weakly equivalent to the square on the left.  The square on
the right is a pullback square of simplicial sets, and by
Proposition~\ref{prophocolimb}, (i) implies that the bottom horizontal
map is a universal simplicial quasifibration.  We conclude that the
square is homotopy cartesian.
\end{proof}

\section[Homotopy calculi of fractions]%
{Homotopy calculi of fractions and mapping cylinders}\label{sechclf}

In this section, we describe the second technical device 
essential to the proof of the main theorems, the homotopy calculi of
fraction introduced in \cite{DKHammock}.  When a category with weak
equivalences admits a homotopy calculus of fractions, the Dwyer-Kan
function complexes $L\aC(A,B)$ admit significantly smaller models
that are nerves of categories of words of a specified type.  We begin
the section with a concise review of this theory.  We then prove that
a Waldhausen category with \hbox{FMCWC} admits a homotopy calculus of left
fractions.  (See Appendix~\ref{appmain} for statements and proofs in
the non-functorial case.)

We begin with the notation for the categories of words of specified
types.  
Let $\aC$ be a category with a subcategory $\w\aC$ of weak
equivalences.  We consider the words on letters $\C$, $\W$, and
$\Wi$:  To every such word $\Word$ and pair of objects $A,B$ in
$\aC$, we associate a category $\Word (A,B)$, where the objects are
roughly speaking words in $\aC$ of the type specified by the letters
in $\Word$ and the morphisms are the weak equivalences.  The
precise definition is as follows. 

\begin{defn}
Let $\Word$ be a word of length $n$ on letters $\C$, $\W$, and
$\Wi$, and let $A,B$ be objects of $\aC$.  We define the $\Word (A,B)$ to
be the following category. An object in $\Word (A,B)$ consists
of:
\begin{itemize}
\item A collection of objects $X_{1},\dotsc,X_{n-1}$ of $\aC$, 
\item A map $f_{i}\colon X_{i}\to X_{i-1}$ in $\aC$ whenever the $i$-th letter of
$\Word$ is $\C$ for some $1\leq i\leq n$,
\item A map $f_{i}\colon X_{i}\to X_{i-1}$ in $\w\aC$ whenever the $i$-th letter of
$\Word$ is $\W$ for some $1\leq i\leq n$, and
\item A map $f_{i}\colon X_{i-1}\to X_{i}$ in $\w\aC$ whenever the $i$-th letter of
$\Word$ is $\Wi$ for some $1\leq i\leq n$,
\end{itemize}
where we interpret $X_{n}$ as $A$ and $X_{0}$ as $B$ in the conditions
above.  A morphism in $\Word (A,B)$ from $\{X_{i},f_{i}\}$ to
$\{X'_{i},f'_{i}\}$ is a collection of maps $g_{i}\colon X_{i}\to
X'_{i}$ in $\w\aC$ that are the identity on $A$ and $B$ and make the
evident diagram commute.
\[
\xymatrix@R-\wordsquish{%
&X_{n-1}\ar@{{}-{}}[r]^{f_{n-1}}\ar[dd]_{\htp}^{g_{n-1}}
&\dotsb\ar@{{}-{}}[r]^{f_{2}}\ar[dd]_{\htp}
&X_{1}\ar@{{}-{}}[dr]^{f_{1}}\ar[dd]_{\htp}^{g_{1}}\\
A\ar@{{}-{}}[ur]^{f_{n}}\ar@{{}-{}}[dr]_{f'_{n}}&&&&B\\
&X'_{n-1}\ar@{{}-{}}[r]_{f'_{n-1}}
&\dotsb\ar@{{}-{}}[r]_{f'_{2}}&X'_{1}\ar@{{}-{}}[ur]_{f'_{1}}
}
\]
\end{defn}

The numbering, which may seem unusual in the diagrams, is forced by the
convention that the letters in the word follow composition order,
where we think of the object $\{X_{i},f_{i}\}$ as representing a
formal composition $f_{1}\circ \dotsb \circ f_{n}$ with the $i$-th map
corresponding to the $i$-th letter (with $f_{i}^{-1}$ in place of
$f_{i}$ in the formal
composition when the $i$-th letter is $\Wi$).  Our words
indicate the same categories and diagrams as those in
\cite{DKHammock}, which are numbered slightly differently.

In our work below, the three most important words are $\WC$, $\WW$,
and $\WCW$.  For convenience, we spell out these categories explicitly.

\begin{ex}\label{exwords}The categories of words $\WC$, $\WW$, and $\WCW$.
\begin{enumerate}
\item $\WC$: An object $\{X,f_{1},f_{2}\}$ is pictured on the left and
a map from $\{X,f_{1},f_{2}\}$ to $\{X',f'_{1},f'_{2}\}$ is pictured
on the right.
\[
\xymatrix@R-\wordsquish{%
&&&&&X\ar[dd]_{\htp}\\
A\ar[r]^{f_{2}}&X&B\ar[l]_{f_{1}}^{\htp}
&&A\ar[ur]^{f_{2}}\ar[dr]_{f'_{2}}
&&B\ar[ul]_{f_{1}}^{\htp}\ar[dl]^{f'_{1}}_{\htp}\\
&&&&&X'
}
\]
\item $\WW$: Objects and maps look the same, but the map $A\to X$ is
required to be in $\w\aC$.
\item $\WCW$: An object $\{X_{1},X_{2},f_{1},f_{2},f_{3}\}$ is pictured on the left and
a map from $\{X_{1},X_{2},f_{1},f_{2},f_{3}\}$ to
$\{X'_{1},X'_{2},f'_{1},f'_{2},f'_{3}\}$ is pictured 
on the right.
\[
\xymatrix@R-\wordsquish{%
&&&&&X_{2}\ar[dd]_{\htp}\ar[r]^{f_{2}}\ar[dl]_{f_{3}}^{\htp}
&X_{1}\ar[dd]_{\htp}\\
A&X_{2}\ar[r]^{f_{2}}\ar[l]_{f_{3}}^{\htp}&X_{1}&B\ar[l]_{f_{1}}^{\htp}
&A
&&&B\ar[ul]_{f_{1}}^{\htp}\ar[dl]^{f'_{1}}_{\htp}\\
&&&&&X'_{2}\ar[r]_{f'_{2}}\ar[ul]^{f'_{3}}_{\htp}&X'_{1}
}
\]
\end{enumerate}
\end{ex}

As described in \cite[5.5]{DKHammock}, the function complex
$L\aC(A,B)$ is a colimit of the nerves of these categories of words.  The
hypothesis of a homotopy
calculus of fractions is a homotopical requirement on how these nerves
fit together.   The following
definition is \cite[6.1]{DKHammock}.  Although we use only homotopy
calculus of left fractions and homotopy calculus of two-sided
fractions in our work below, we include the definition of homotopy
calculus of right fractions for completeness. 

\begin{defn}[Homotopy calculi of fractions]
Let $\aC$ be a category with a subcategory of weak equivalences
$\w\aC$.  
\begin{enumerate}
\item $\aC$ admits a \term{homotopy calculus of 
two-sided fractions (HC2F)} means that for every pair of integers $i,j \geq
0$, the functors given by inserting an identity morphism in the
$(i+1)$-st spot,
\begin{gather*}
\Wi \C^{i+j} \Wi \to \Wi \C^{i} \W^{-1} \C^{j} \Wi \\
\Wi \W^{i+j} \Wi \to \Wi \W^{i} \Wi \W^{j} \Wi,
\end{gather*}
induce weak equivalences on nerves for every pair of objects of $\aC$.
\item $\aC$ admits a \term{homotopy calculus of left
fractions (HCLF)} means that for every pair of integers $i,j \geq 0$, 
the functors given by inserting an identity morphism in the
$(i+1)$-st spot,
\begin{gather*}
\Wi\C^{i+j} \to \Wi \C^{i} \Wi \C^{j} \\
\Wi\W^{i+j} \to \Wi \W^{i} \Wi \W^{j},
\end{gather*}
induce weak equivalences on nerves for every pair of objects of $\aC$.
\item$\aC$ admits a \term{homotopy calculus of right
fractions (HCRF)} means that for every pair of integers $i,j \geq 0$, 
the functors given by inserting an identity morphism in the
$i$-th spot,
\begin{gather*}
\C^{i+j} \Wi \to \C^{i} \Wi \C^{j} \Wi \\ 
\W^{i+j} \Wi \to \W^{i} \Wi \W^{j} \Wi,
\end{gather*}
induce weak equivalences on nerves for every pair of objects of $\aC$.
\end{enumerate}
\end{defn}

Dwyer and Kan observe \cite[6.1,\S9]{DKHammock} 
that if $\aC$ admits a homotopy calculus of left or right fractions, then
$\aC$ admits a homotopy calculus of two-sided fractions.  The following
proposition \cite[6.2]{DKHammock} explains the utility and terminology
of the definition.

\begin{prop}
Let $\aC$ be a category with a subcategory of weak equivalences
$\w\aC$.  
\begin{enumerate}
\item If $\aC$ admits a homotopy calculus of two-sided
fractions, then the maps 
\begin{gather*}
N\WCW (A,B) \to L\aC(A,B) \\
N\Wi \W \Wi (A,B) \to L(\w\aC)(A,B)
\end{gather*}
are weak equivalences.

\item If $\aC$ admits a homotopy calculus of left
fractions, then the maps
\begin{gather*}
N\WC (A,B) \to L\aC(A,B) \\
N\WW (A,B) \to L(\w\aC)(A,B)
\end{gather*}
are weak equivalences.
\item If $\aC$ admits a homotopy calculus of right
fractions, then the maps
\begin{gather*}
N\C \Wi (A,B) \to L\aC(A,B) \\
N\W \Wi (A,B) \to L(\w\aC)(A,B)
\end{gather*}
are weak equivalences.
\end{enumerate}
\end{prop}

The following theorem, the main theorem of this section, is proved
below.  In it, $\wco\aC$ 
denotes the category of weak cofibrations
(q.v. \cite[2.4]{BlumbergMandell}), which by definition contains the
weak equivalences $\w\aC$ as a subcategory.

\begin{thm}\label{thmHCLF}
Let $\aC$ be a Waldhausen category with FMCWC.  Then $\aC$,
$\wco\aC$, and $\w\aC$ admit homotopy
calculi of left fractions.
\end{thm}

Under the hypothesis of FMCWC, the previous theorem in particular
allows us to model the function complex $L\aC(A,B)$ in terms of the
categories of words $\WC(A,B)$ and $\WCW(A,B)$.  The following theorem
identifies the subsets  $\Lco\aC(A,B)$ and $\Lw\aC(A,B)$ in these
terms.  In it, for $\Word =\WC$ and $\Word=\WCW$, we write
$\Word (A,B)_{\wco}$ and $\Word (A,B)_{\w}$ for the full
subcategory of $\Word(A,B)$ of diagrams whose forward map is a weak
cofibration and weak equivalence, respectively.

\begin{thm}\label{thmwordcomp}
Let $\aC$ be a Waldhausen category with FMCWC.  Then the maps
\begin{gather*}
\WC(A,B)_{\wco}\to \Lco\aC(A,B), \qquad 
\WCW(A,B)_{\wco}\to \Lco\aC(A,B),\\
\WC(A,B)_{\w}\to \Lw\aC(A,B),\qquad\text{and}\qquad 
\WCW(A,B)_{\w}\to \Lw\aC(A,B)
\end{gather*}
are weak equivalences
\end{thm}

\begin{proof}
We do the case for the weak cofibrations; the case for the weak
equivalences is identical.  Since $\WC(A,B)_{\wco}$ and
$\WCW(A,B)_{\wco}$ are collections of components of the categories
$\WC(A,B)$ and $\WCW(A,B)$, and $\Lco\aC(A,B)$ is a collection of
components of the simplicial set $L\aC(A,B)$, we just need to check that
these components coincide.  For this, it suffices to show that every
component of $\Lco\aC(A,B)$ contains as a vertex the image of an
object of $\WC(A,B)_{\wco}$.  The key observation is that the category
$\WC(A,B)_{\wco}$ is precisely the category of words $\WC(A,B)$ in
the category $\wco\aC$.  By definition, a component of $\Lco\aC(A,B)$
contains a vertex represented by a zigzag in which all the forward
arrows are weak cofibrations.  Interpreting this zigzag as a vertex in
$L(\wco\aC)(A,B)$, we get a path in $L(\wco\aC)(A,B)$ to a vertex in
the image of $\WC(A,B)_{\wco}$ since $\wco\aC$ admits a homotopy
calculus of left fractions.  The inclusion of $\wco\aC$ induces a map
$L(\wco\aC)(A,B)\to \Lco\aC(A,B)$ that gives us a path in
$\Lco\aC(A,B)$ to a vertex in the image of $\WC(A,B)_{\wco}$.
\end{proof}

Using the previous theorem and the observation in its proof
that $\WC(A,B)_{\wco}$ and $\WC(A,B)_{\w}$ coincide with the categories of
words $\WC(A,B)$ for $\wco\aC$ and $\w\aC$ (respectively), we
obtain the following corollary.

\begin{cor}\label{corcomp}
For any $A,B$ in $\aC$, the inclusions $L(\wco\aC)(A,B)\to
\Lco\aC(A,B)$ and $L(\w\aC)(A,B)\to \Lw\aC(A,B)$ are weak
equivalences. 
\end{cor}

Likewise, we obtain the proof of Proposition~\ref{prophypii}, which was
used in the proof of Theorem~\ref{main}:

\begin{cor}\label{corhypii}
In the context of Theorem~\ref{main}, hypothesis~(ii) implies hypothesis~(iii).
\end{cor}

\begin{proof}
Suppose $\phi$ is in the image of $\Ho\w\aD(FA,FB)$ in $\Ho\aD(FA,FB)$.
Since $F$ induces an equivalence $\Ho\aC\to \Ho\aD$,
$\phi$ is represented by $F\phi'$ for some $\phi'$ in
$\Ho\aC(A,B)$.  We can represent $\phi'$ by a word
\[
\xymatrix@-1pc{%
A\ar[r]^{f_{2}}&X&B\ar[l]_{f_{1}}^{\htp}
}
\]
in
$\WC(A,B)$ for $\aC$, and then $\phi$ is represented by the word
\[
\xymatrix@-1pc{%
FA\ar[r]^{Ff_{2}}&X&B\ar[l]_{Ff_{1}}^{\htp}
}
\]
in
$\WC(FA,FB)$ for $\aD$.  Theorem~\ref{thmwordcomp} implies that 
this word is in the components $\WC(FA,FB)_{\w}$ of $\WC(FA,FB)$ and
hence that $Ff_{2}$ is a weak 
equivalence.  Hypothesis~(ii) implies that $f_{2}$ is a
weak equivalence in $\aC$, and it follows that $\phi$ is in the image
of $\Ho\w\aC(A,B)$. 
\end{proof}

The remainder of the section is devoted to the proof of
Theorem~\ref{thmHCLF}.  Thus, assume that $\aC$ has FMCWC and
fix $A,B$ in $\aC$; we
need to prove that for every pair of 
integers $i,j\geq 0$, the functor $\Wi\C^{i}\C^{j}(A,B)\to
\Wi\C^{i}\Wi\C^{j}(A,B)$ induces a weak equivalence of nerves 
in each of the three cases where the  letter ``$\C$'' in the words
indicates the category 
$\aC$, $\wco\aC$, or $\w\aC$.  The proof is the same in all three
cases.   The following lemma is the case $i=0$.

\begin{lem}\label{lemtriv}
The functor $\Wi\C^{j}(A,B)\to \Wi\Wi\C^{j}(A,B)$ induces a homotopy equivalence
of nerves.
\end{lem}

\begin{proof}
Composition induces a functor back $\Wi\Wi\C^{j}(A,B)\to \Wi\C^{j}(A,B)$.  The
composite functor on $\Wi\C^{j}(A,B)$ is the identity and the
composite functor on
$\Wi\Wi\C^{j}(A,B)$ has a natural transformation to the identity.
\[
\xymatrix@R-\wordsquish{%
&X_{j+1}\ar[r]\ar[dd]^{=}&\dotsb\ar[dd]^{=}\ar[r]
&X_{2}\ar[dd]^{=}&X_{1}\ar[l]^{\htp}_{f_{2}}\ar[dd]_{f_{2}}\\
A\ar[ur]\ar[dr]&&&&&B\ar[ul]^{\htp}_{f_{1}}\ar[dl]_{\htp}^{f_{2}\circ f_{1}}\\
&X_{j+1}\ar[r]&\dotsb\ar[r]
&X_{2}&X_{2}\ar[l]^{=}
}
\]
These functors then induce inverse homotopy equivalences on nerves.
\end{proof}

The following lemma now completes the proof of Theorem~\ref{thmHCLF}.

\begin{lem}\label{lemHCLF}
For $i>0$ and $j\geq 0$, the functor
\[
\Wi \C^{i}\C^{j}(A,B) \to \Wi \C^i \Wi \C^j(A,B)
\]
induces a weak equivalence on
nerves. 
\end{lem}

\begin{proof}
We obtain a functor $\Wi \C^i \Wi \C^j(A,B)\to \Wi\C^{i}\C^{j}(A,B)$
by applying the mapping cylinder functor and taking
pushouts:  The object
\[
\xymatrix@-1pc{%
A\ar[r]&Y_{j}\ar[r]&\dotsb\ar[r]&Y_{1}&Z\ar[l]_{f}^{\htp}\ar[r]
&X_{i}\ar[r]&\dotsb\ar[r]&X_{1}&B\ar[l]_(.4){\htp}
}
\]
of $\Wi \C^i \Wi \C^j(A,B)$ is sent to the object
\[
\xymatrix@-1pc{%
A\ar[r]&Y_{j}\ar[r]&\dotsb\ar[r]&Y_{1}\ar[r]&Tf\cup_{Z}X_{i}\ar[r]
&Tf\cup_{Z} X_{i-1}\ar[r]&\dotsb\ar[r]&Tf\cup_{Z}X_{1}&B\ar[l]_(.3){\htp}
}
\]
of $\Wi\C^{i}\C^{j}(A,B)$.  
The composite functor on $\Wi \C^i
\Wi \C^j(A,B)$ has a zigzag of natural transformations relating it to
the identity:
\[
\xymatrix@C-1pc{%
&&Y_{j}\ar[r]\ar[d]_{=}&\dotsb\ar[r]\ar[d]_{=}&Y_{1}\ar[d]_{=}
&Z\ar[l]_{\htp}\ar[r]\ar@{ >->}[d]_{\htp}
&X_{i}\ar[r]\ar[d]_{\htp}&\dotsb\ar[r]\ar[d]_{\htp}&X_{1}\ar[d]_{\htp}\\
A\ar[urr]\ar[drr]\ar[rr]&&Y_{j}\ar[r]&\dotsb\ar[r]&Y_{1}&Tf\ar[l]^{\htp}\ar[r]
&Tf\cup_{Z} X_{i}\ar[r]
&\dotsb\ar[r]&Tf\cup_{Z}X_{1}
&&B.\ar[ull]_{\htp}\ar[dll]^{\htp}\ar[ll]_{\htp}\\
&&Y_{j}\ar[r]\ar[u]^{=}&\dotsb\ar[r]\ar[u]^{=}
&Y_{1}\ar[u]^{=}&Y_{1}\ar[l]^{=}\ar[r]\ar[u]^{\htp}
&Tf\cup_{Z} X_{i}\ar[r]\ar[u]^{=}
&\dotsb\ar[r]\ar[u]^{=}&Tf\cup_{Z}X_{1}\ar[u]^{=}
}
\]
For the composite functor on
$\Wi\C^{i}\C^{j}(A,B)$,  $f$ is the identity map $Z=Y_{1}$, and the map
$Tf\to Y_{1}=Z$ composed with the map $Z\to X_{i}$ induces a natural
transformation from the composite 
functor to the identity
\[
\xymatrix@C-1pc@R-1pc
{%
&Y_{j}\ar[r]\ar[dd]_{=}&\dotsb\ar[dd]_{=}\ar[r]
&Y_{1}\ar[dd]_{=}\ar[r]
&Tf\cup_{Z} X_{i}\ar[r]\ar[dd]_{\htp}
&\dotsb\ar[r]\ar[dd]_{\htp}&Tf\cup_{Z}X_{1}\ar[dd]_{\htp}\\
A\ar[ur]\ar[dr]&&&&\qquad&&&B.\ar[ul]_{\htp}\ar[dl]^{\htp}\\
&Y_{j}\ar[r]&\dotsb\ar[r]&Y_{1}\ar[r]
&X_{i}\ar[r]&\dotsb\ar[r]&X_{1}
}
\]
The induced map on nerves is therefore a generalized simplicial
homotopy equivalence.
\end{proof}

\section[Homotopy cocartesian squares]%
{Homotopy cocartesian squares in {W}aldhausen categories}\label{sechococart}

Fundamentally, algebraic $K$-theory is about splitting ``extensions'',
and in Waldhausen's framework, the category of cofibrations specifies
the extensions to split. 
The key concept is the homotopy cocartesian square, 
i.e., a square diagram that is weakly equivalent (by a
zigzag) to a pushout square where one of the parallel sets of arrows
consists of cofibrations.  In simplicial model categories, the
homotopy cocartesian squares can be characterized in terms of mapping
spaces.  The Dwyer-Kan simplicial localization and function complexes
extend this alternative definition of homotopy cocartesian to the context
of Waldhausen categories.  In this section we show that under mild
hypotheses these two definitions are equivalent in Waldhausen
categories.  This equivalence reflects another aspect of the intrinsic
relationship between the Dwyer-Kan localization and algebraic
$K$-theory, and plays a key role in the proofs of Theorems~\ref{main} 
and~\ref{thmhh}.  The following is the main theorem of the section.

\begin{thm}\label{thmconhococart}
Let $\aC$ be a Waldhausen category whose weak
equivalences are closed under retracts. Suppose furthermore that $\aC$
admits a HCLF and every
weak cofibration in $\aC$ has a mapping cylinder (e.g., when $\aC$
has FMCWC). Let $A\to B$ be a weak
cofibration and $A\to C$, $B\to D$, and $C\to D$ maps that make the
square on the left commute. 
\[
\xymatrix@-1pc{%
A\ar[r]^{wc}\ar[d]&B\ar[d]
&&L\aC(D,E)\ar[r]\ar[d]&L\aC(C,E)\ar[d]\\
C\ar[r]&D
&&L\aC(B,E)\ar[r]&L\aC(A,E)
}
\]
Then the square on the left is homotopy cocartesian in $\aC$ if and
only if for every $E$ the square of simplicial sets on the right is
homotopy cartesian.
\end{thm}

In fact, the forward direction holds under slightly weaker
hypotheses, and we state this as the following theorem, which implies
Theorem~\ref{thmprehococart} in Section~\ref{secpfmain}.   A similar
result is the main theorem of \cite{WeissHammock}. 

\begin{thm}\label{thmhococart}
Let $\aC$ be a Waldhausen category that admits a HCLF. 
For a cofibration $A\to B$, a map $A\to C$, $D=B\cup_{A}C$, and any
object $E$, the following square is homotopy cartesian:
\[
\xymatrix@-1pc{%
L\aC(D,E)\ar[r]\ar[d]&L\aC(C,E)\ar[d]\\
L\aC(B,E)\ar[r]&L\aC(A,E)
}
\]
\end{thm}

The previous theorems make it easier to check in certain cases that
functors are weakly exact.  Applying Theorem~\ref{thmhococart} in
$\aC$ and Theorem~\ref{thmconhococart} in $\aD$ gives the following
corollary. 

\begin{cor}\label{corweakexact}
Let $\aC$ be a Waldhausen category that admits a $HCLF$, and let
$\aD$ be a Waldhausen category that admits a HCLF, whose weak
equivalences are closed under retracts, and whose weak cofibrations
have mapping cylinders.  Let $F\colon \aC\to \aD$ be a
functor that preserves weak equivalences and weak cofibrations.  If
$F$ is a DK-equivalence, then $F$ preserves homotopy cocartesian
squares and so is weakly exact.
\end{cor}


The hypothesis that weak equivalences are closed under retracts is
familiar from the theory of model categories.  Two other properties of
weak equivalences in model categories are currently somewhat less
well-known but explored in \cite{DKHS}.  The more subtle of these is
the ``two out of six''  property \cite[7.3]{DKHS}, which we abbreviate
to \term{DKHS-2/6}.   The subcategory $\w\aC$ satisfies
DKHS-2/6 when for any three composable maps
\[
f\colon A\to B, \qquad g\colon B\to C, \qquad h\colon C\to D,
\]
if the composites $g\circ f$ and $h\circ g$ are in $\w\aC$, then so are
the original maps $f$, $g$, and $h$. The proof of
Theorem~\ref{thmconhococart} depends more directly on the other
property, which \cite[8.4]{DKHS} calls ``saturated'' and we call 
\term{DKHS-saturated} (to avoid confusion with Waldhausen's
terminology).  By definition, the localization functor $\aC\to
\Ho\aC$ sends weak equivalences to isomorphisms.  We say that the weak
equivalences of $\aC$ are DKHS-saturated when a map is a weak
equivalence if and only if its image in $\Ho\aC$ is an isomorphism.
Note that when $\aC$ admits a homotopy calculus of two-sided
fractions, the weak equivalences of $\aC$ are DKHS-saturated if and
only if the subcategory of $\Ho\aC$ generated by the weak equivalences
and their inverses consists of all the isomorphisms of $\Ho\aC$.
We prove the following theorem at the end of the section.

\begin{thm}\label{thmTFAE}
Let $\aC$ be a Waldhausen category that admits a HCLF and assume that every
weak cofibration in $\aC$ has a mapping cylinder.  The following are equivalent:
\begin{enumerate}
\item The weak equivalences are closed under retracts.
\item The weak equivalences satisfy the DKHS-2/6 property.
\item The weak equivalences are DKHS-saturated.
\end{enumerate}
\end{thm}

We can now prove Theorem~\ref{thmconhococart}, assuming
Theorems~\ref{thmhococart} and~\ref{thmTFAE}.

\begin{proof}[Proof of Theorem~\ref{thmconhococart} from
Theorems~\ref{thmhococart} and~\ref{thmTFAE}]
The ``only if'' direction follows from Theorem~\ref{thmhococart}.  
For the ``if'' direction, by factoring
the weak cofibration $A\to B$ as a cofibration followed by a weak
equivalence, it suffices to consider the case when $A\to B$ is a
cofibration.  Consider the induced map $B\cup_{A}C\to D$; in the
commutative diagram of simplicial sets 
\[
\xymatrix@-1pc{%
L\aC(D,E)\ar[r]\ar[dr]
&L\aC(B\cup_{A}C,E)\ar[r]\ar[d]&L\aC(C,E)\ar[d]\\
&L\aC(B,E)\ar[r]&L\aC(A,E)
}
\]
both the outer ``square'' and the inner square are homotopy cartesian.
It follows that the map $L\aC(D,E)\to L\aC(B\cup_{A}C,E)$ is a weak
equivalence for all objects $E$, and so in particular, $B\cup_{A}C\to
D$ is an isomorphism in $\Ho\aC$.  Because $\aC$ is DKHS-saturated by
Theorem~\ref{thmTFAE}, $B\cup_{A}C\to D$ is a weak
equivalence.
\end{proof}

The proof of Theorem~\ref{thmhococart} is slightly more
complicated. 
We apply Quillen's Theorem~B as formulated in Theorem~\ref{thmC}
to the short hammock version of $L\aC(B,E)$, the nerve of the category
$\WC(A,E)$.  Recall that $\WC(A,E)$ is the
category whose objects are the zigzags $\overleftrightarrow{X}$
\[
\xymatrix@-1pc{%
A\ar[r]&X&E\ar[l]_{\htp}
}
\]
and whose maps are the maps $X\to X'$ under $A$ and $E$, 
as in Example~\ref{exwords}.(i).
Composition with $f\colon A\to B$ induces a functor
$f^{*}\colon \WC(B,E)\to\WC(A,E)$.  Theorem~\ref{thmhococart} is 
an immediate consequence of Theorem~\ref{thmC} and the following
lemma.

\begin{lem}\label{lemthmhococart}
Assume $\aC$ admits a HCLF and $f\colon A\to B$ is a cofibration.
For any map $\overleftrightarrow{X}\to \overleftrightarrow{X'}$ in $\WC(A,E)$ the induced map
of comma categories from $\overleftrightarrow{X'}\downarrow f^{*}$ to
$\overleftrightarrow{X}\downarrow f^{*}$ induces a weak equivalence of nerves.
\end{lem}

\begin{proof}
The argument is another application of Theorem~\ref{thmC}.
Let $\w\aC_{E}$ denote the full
subcategory of $\w\aC$ consisting of those objects that are weakly
equivalent to $E$.  We have a functor $G_{\overleftrightarrow{X}}$ from
$\overleftrightarrow{X}\downarrow f^{*}$ to $\w\aC_{E}$ that sends the object
$\overrightarrow{Y}$
\[
\xymatrix@-1pc{%
A\ar[r]\ar@{ >->}[d]&X\ar[d]^{\htp}&E\ar[l]_{\htp}\\
B\ar[r]&Y
}
\]
to $Y$.  For a map $\overleftrightarrow{X}\to\overleftrightarrow{X'}$
in $\WC(A,E)$, consider the following
strictly commuting diagram of functors:
\begin{equation}\label{eqxxpsq}
\begin{gathered}
\xymatrix{%
\overleftrightarrow{X'}\downarrow f^{*}\ar[r]^{G_{\overleftrightarrow{X'}}}\ar[d]
&\w\aC_{E}\ar[d]^{=}\\
\overleftrightarrow{X}\downarrow f^{*}\ar[r]_{G_{\overleftrightarrow{X}}}
&\w\aC_{E}.
}
\end{gathered}
\end{equation}
We verify that this diagram satisfies the hypotheses of
Theorem~\ref{thmC}.
For any object $H$ in $\w\aC_{E}$, the
comma category $H\downarrow G_{\overleftrightarrow{X}}$ has as objects 
the diagrams of the form
\[
\xymatrix@-1pc{%
A\ar[r]\ar@{ >->}[d]&X\ar[d]^{\htp}&E\ar[l]_{\htp}\\
B\ar[r]&Y&H.\ar[l]^{\htp}
}
\]
Using the universal property of the pushout, we see that this category
is equivalent to the full subcategory of $\WC(B\cup_{A}X,H)$ of those
zigzags 
\[
\xymatrix@-1pc{%
B\cup_{A}X\ar[r]&Y&H\ar[l]_{\htp}
}
\]
for which the composite map $E\to B\cup_{A}X\to Y$ is a weak
equivalence.  Thus, $H\downarrow G_{\overleftrightarrow{X}}$ is equivalent to a
disjoint union of certain components of
$\WC(B\cup_{A}X,H)$. Hypotheses (i) and (ii) of Theorem~\ref{thmC}
follow from the fact that
$N(\WC(B\cup_{A}X,H))$ preserves weak equivalences in $H$ and $X$.
We conclude that diagram~\ref{eqxxpsq} is a homotopy cartesian square.
Since the
vertical map on the right is a weak equivalence, so is the vertical map on
the left.
\end{proof}

It remains to prove Theorem~\ref{thmTFAE}.
Obviously (iii) implies both (i) and (ii); we show that (i) implies
(ii) and (ii) implies (iii).

\begin{proof}[Proof that (i) implies (ii)]
Let $A\to B\to C\to D$ be a sequence of composable maps, with $A\to C$
and $B\to D$ weak equivalences.  Since $A\to C$ is a weak equivalence,
it is a weak cofibration, and so we can factor it as a cofibration
$A\to C'$ followed by a split weak equivalence $C'\to C$. Let
$B'=C'\cup_{A}B$.
\[
\xymatrix{%
A\ar[r]^{\htp}\ar@{ >->}[d]_{\htp}&C
&&A\ar@{ >->}[r]^{\htp}\ar[d]&C'\ar[d]\\
C'\ar[ur]_{\htp}&C\ar[l]\ar[u]_{=}
&&B\ar@{ >->}[r]_{\htp}&B'
}
\]
We have a composite map $f\colon C\to C'\to B'$, and the compatible
maps $C'\to C$ and $B\to C$ induce a map $g\colon B'\to C$ such that
$g\circ f$ is the identity on $C$.  Since the composite map $B\to C\to
D$ is a weak equivalence, the composite of $g$ with $C\to D$ is a weak
equivalence.   We therefore obtain a commutative diagram
\[
\xymatrix{%
C\ar[r]^{f}\ar[d]&B'\ar[r]^{g}\ar[d]_{\htp}&C\ar[d]\\
D\ar[r]_{=}&D\ar[r]_{=}&D,
}
\]
where both horizontal composites are the identity and the middle
vertical map is a weak equivalence.  We conclude from (i) that the
map $C\to D$ is a weak equivalence, and it follows that $B\to C$ and
$A\to B$ are weak equivalences. 
\end{proof}

\begin{proof}[Proof that (ii) implies (iii)](cf. \cite[36.4]{DKHS})
Let $a\colon A\to B$ be a map
in $\aC$ that becomes an isomorphism in $\Ho\aC$.  Since
$\aC$ admits a HCLF, the inverse isomorphism from $\Ho\aC$ is
represented by a zigzag (in $\aC$) of the form
\begin{equation}\label{eqzigzagBCA}
\begin{gathered}
\xymatrix{%
B\ar[r]^{b}&C&A\ar[l]^{\htp}_{c}
}
\end{gathered}
\end{equation}
for some $C$.  Moreover, using a mapping cylinder, we can assume without
loss of generality that $c\colon A\to C$ is a cofibration as well as a
weak equivalence.  The composite zigzag
\[
\xymatrix{%
A\ar[r]^{b\circ a}&C&A\ar[l]^{\htp}_{c}
}
\]
is in the component of the identity on $A$, and so $b\circ
a$ is a weak equivalence. Let $B'=B\cup_{A}C$, and let $C'=C\cup_{B}B'$.
\begin{equation}\label{eqdiagiiitoi}
\begin{gathered}
\xymatrix{%
&A\ar@{ >->}[d]_{c}^{\htp}\ar[r]^{a}
&B\ar@{ >->}[d]^{\htp}\ar[r]^{b}
&C\ar@{ >->}[d]^{\htp}\\
B\ar[r]_{b}&C\ar[r]&B'\ar[r]&C'
}
\end{gathered}
\end{equation}
The composite $C\to C'$ is a weak equivalence because it is the
pushout of the weak equivalence $b\circ a$ over the cofibration $c$.
The zigzag 
\[
\xymatrix{%
B\ar[r]_{b}&C\ar[r]&B'&B\ar@{ >->}[l]_{\htp}
}
\]
is in the component representing the composite of the
zigzag~\eqref{eqzigzagBCA} with $a$, i.e., the component containing
the identity of $B$.  It follows that in the
diagram~\eqref{eqdiagiiitoi}, the horizontal composite map $B\to B'$
is a weak equivalence.  Applying DKHS-2/6 to the bottom horizontal
sequence of maps in~\eqref{eqdiagiiitoi}, we conclude that $b$ is a weak
equivalence, and hence that $a$ is a weak equivalence.
\end{proof}

\section{Proof of Theorems~\ref{thmhocoendone}, \ref{thmhocoendtwo},
and~\ref{thmhocoend}}
\label{secpdecomp}

We begin with the proof of part~(i) of Theorem~\ref{thmhocoend}.  The
first reduction is to replace $\w\Spdot[n] \aC$ with a simpler
category.  Let $F'_{n-1}\aC$ denote the Waldhausen category whose
objects are  the sequences of $n-1$ composable weak cofibrations in
$\aC$.  We have an exact forgetful functor $\Spdot[n] \aC \to
F'_{n-1} \aC$ that sends an object $\{A_{i,j}\}$ of $\Spdot[n]\aC$ to
the sequence 
\[
A_{0,1}\to A_{0,2} \to \dotsb \to A_{0,n}.
\]
The Waldhausen category $F'_{n-1}\aC$ is the analogue for $\Spdot\aC$
of the Waldhausen category 
$F_{n-1}\aC$, whose objects are the sequences of $n-1$
composable cofibrations in $\aC$.  The forgetful functor $\Sdot[n]\aC\to
F_{n-1}\aC$ is exact and an equivalence of categories, whose inverse
equivalence is also exact.  Proposition~\ref{propbm}.(iii) and the
analogous fact for $F_{n-1}$ and $F'_{n-1}$  then implies the
following proposition.  (See Appendix~\ref{appmain} for the statement
in the non-functorial case.)

\begin{prop}\label{propfprime}
If $\aC$ admits FFWC, the forgetful functor $\w \Spdot[n]\aC\to
\w F'_{n-1}\aC$ induces a weak equivalence of nerves.
\end{prop}

We use proposition~\ref{propfprime} to simplify one side of the
equivalence in part~(i) of Theorem~\ref{thmhocoend}, and we use
homotopy calculus 
of left fractions to simplify the other side.  Although the
categories $\WC$ produce much more manageable simplicial sets than the
hammock function complexes, they are contravariant in weak
equivalences of each variable and so do not have the right
functoriality to fit into a homotopy coend.  The categories $\WCW$ are
covariant in weak equivalences of the source variable and
contravariant in weak equivalences of the target variable, which is
the opposite variance expected of a function complex.  We do likewise
have such an opposite variance on the hammock function complexes
$L\aC(X,Y)$ since the category $L\aC$ contains ``backward'' copies of
the weak equivalences.  The following lemma compares the homotopy
coend in Theorem~\ref{thmhocoend} with the homotopy coend for the
opposite variance.

\begin{lem}\label{lemreversi}
When $\aC$ satisfies HCLF, 
\begin{gather*}
\hocoendlim_{(X_{1},\dotsc,X_{n})\in \w\aC^{n}}\Lco\aC(X_{n-1},X_{n})\times \dotsb \times
\Lco\aC(X_{1},X_{2})\\
\noalign{and}
\hocoendlim_{(X'_{1},\dotsc,X'_{n})\in (\w\aC^{\op})^{n}}\Lco\aC(X'_{n-1},X'_{n})
\times \dotsb \times
\Lco\aC(X'_{1},X'_{2})
\end{gather*}
are weakly equivalent.
\end{lem}

\begin{proof}
Write $B$ and $B'$ for these homotopy coends, and let $D$ be the
homotopy coend
\begin{multline*}
D=\hocoend_{(X_{1},X'_{1},\dotsc,X_{n},X'_{n})\in  (\w\aC\times \w\aC^{\op})^{n}}\\
\WW(X'_{n},X_{n})\times \Lco\aC(X'_{n-1},X_{n})\times \WW(X_{n-1},X'_{n-1})\times\dotsb  \\
\dotsb \times \WW(X_{2},X'_{2})\times \Lco\aC(X'_{1},X_{2})\times \WW(X_{1},X'_{1}).
\end{multline*}
Composition then induces maps $D\to B$ and $D\to B'$.  Let $C$ be the
homotopy colimit
\begin{multline*}
C=\hocolim_{(X_{1},X'_{1},\dotsc,X_{n},X'_{n})\in  (\w\aC\times \w\aC^{\op})^{n}}\\
\WW(X'_{n},X_{n})\times  \WW(X_{n-1},X'_{n-1})\times\dotsb 
\times \WW(X_{1},X'_{1}).
\end{multline*}
We have an evident map $D\to C$ obtained by dropping the $\Lco\aC$
factors, and we have maps $C\to N\w\aC^{n}$ and $C\to
N(\w\aC^{\op})^{n}$ obtained from the canonical map $C\to
N((\w\aC\times \w\aC^{\op})^{n})$ by dropping the $X'_{i}$ or the
$X_{i}$ respectively.   We then have the following commuting diagrams
\[
\xymatrix{%
D\ar[r]\ar[d]&C\ar[d]&&D\ar[r]\ar[d]&C\ar[d]\\
B\ar[r]&N\w\aC^{n}&&B'\ar[r]&N(\w\aC^{op})^{n}
}
\]
that are easily seen to be pullback squares.  By
Proposition~\ref{prophocolimb}, the canonical map $C\to N((\w\aC\times
\w\aC^{\op})^{n})$ is a universal simplicial quasifibration as are
projection maps, and so the right vertical maps above are universal
simplicial quasifibrations. We will show that for each vertex of
$N\w\aC^{n}$ and of $N(\w\aC^{\op})^{n}$, the fiber of the right
vertical map is weakly contractible; it then follows that the vertical maps
are weak equivalences, and this gives a zigzag of weak equivalences
relating $B$ and $B'$.

Thus, we are reduced to proving that the fibers of the right vertical
maps are weakly contractible.  We will treat the case of 
$N\w\aC^{n}$, the case of $N(\w\aC^{\op})^{n}$ being similar.  The map
$C\to N\w\aC^{n}$ is the product of maps $\hocolim_{X,X'} \WW(X,X')\to
N\w\aC$,
and so it suffices to see that each of these maps has contractible
fiber.  Fixing a vertex $X$ in $N\w\aC$, the 
fiber is the simplicial set with $r$-simplices the diagrams
\[\xymatrix@-1.25ex{%
X\ar[dr]_{\htp}&X'_{r}\ar[d]^{\htp}
&\relax\dotsb \ar[l]_{\htp}&X'_{0}\ar[l]_{\htp}\\
&A_{0}\ar[r]_{\htp}&\relax\dotsb\ar[r]_{\htp}
&A_{r}.
}
\]
We can regard this as the diagonal of the bisimplicial set $F\dsubdot$
with
$(q,r)$-simplices the diagrams
\[\xymatrix@-1.25ex{%
X\ar[dr]_{\htp}&X'_{r}\ar[d]^{\htp}
&\relax\dotsb \ar[l]_{\htp}&X'_{0}\ar[l]_{\htp}\\
&A_{0}\ar[r]_{\htp}&\relax\dotsb\ar[r]_{\htp}
&A_{q}.
}
\]
We have a bisimplicial map to the bisimplicial set $A\dsubdot$ with
$(q,r)$-simplices the diagrams
\[
\xymatrix@-1.25ex{%
X\ar[r]^{\htp}&A_{0}\ar[r]^{\htp}&\relax\dotsb\ar[r]^{\htp}&A_{q}
}
\]
(constant in the $r$ direction) by forgetting the objects
$X'_{0},\dotsc, X'_{r}$.  For each fixed $q$-simplex, this is a
homotopy equivalence in the $r$ direction using the usual simplicial
contraction argument, i.e., using the contraction on $N(\w\aC^{\op}\bs
A_{0})$. It 
follows that the map $F\dsubdot \to A\dsubdot$ is a weak 
equivalence.  On the other hand, $A\dsubdot$ is the constant
bisimplicial set on $N(\w\aC\bs X)$ and so is contractible.  We
conclude that $F\dsubdot$ and its diagonal are weakly contractible.
\end{proof}

We now prove part~(i) of Theorem~\ref{thmhocoend} by comparing
$N(\w F'_{n-1})$ with the homotopy coend in Lemma~\ref{lemreversi}.
Thus, let 
$\aC$ be a Waldhausen category that has FMCWC and fix $n\geq 2$.  We
prove the comparison in a sequence of reductions
$\mathbb{A},\mathbb{B},\mathbb{C}$ obtained from applying simplicial
homotopy theory. 

\begin{lem}
$N\w F'_{n-1}\aC$ is equivalent to the diagonal simplicial set of the
bisimplicial set $\mathbb{A}$ that has as its $(q,r)$ simplices the commutative
diagrams of the following form
\[
\xymatrix@C-1.25ex@R-\redsquish{%
&&&X_{0,1}\ar[r]^{wc}\ar[d]_{\htp}&\dotsb \ar[r]^{wc}\ar[d]_{\htp}
&X_{0,n}\ar[d]^{\htp}
&B_{0}\ar[l]_{\htp}&\dotsb\ar[l]_{\htp}&B_{q}\ar[l]_{\htp}\\
&&&\vdots\ar[r]^{wc}\ar[d]_{\htp}
&\dotsb \ar[r]^{wc}\ar[d]_{\htp}&\vdots\ar[d]^{\htp}\\
A_{q}&\dotsb \ar[l]_{\htp}&A_{0}\ar[l]_{\htp}
&X_{r,1}\ar[l]_{\htp}\ar[r]^{wc}&\dotsb \ar[r]^{wc}&X_{r,n},
}
\]
where the maps labeled ``$\htp$'' are weak equivalences and the maps
labeled ``$wc$'' are weak cofibrations.
\end{lem}

\begin{proof}
Each $(q,r)$ simplex is specified by an $r$-simplex of $N\w F'_{n-1}\aC$,
a $q$-simplex of $N\w(\aC\bs X_{r,1})$ and a $q$-simplex of
$N\w(\aC/X_{0,n})$.  Regarding $N\w F'_{n-1}\aC$ as a bisimplicial set
constant in the $q$ direction, we get a bisimplicial map from
$\mathbb{A}$ to $N\w F'_{n-1}\aC$.  Since for each fixed $r$-simplex of
$N \w F'_{n-1}\aC$, the simplicial sets $N\w(\aC\bs X_{r,1})$ and 
$N\w(\aC/X_{0,n})$ are contractible, the bisimplicial map induces a
weak equivalence on diagonals.
\end{proof}

In the case $n=2$, the diagonal of the bisimplicial set $\mathbb{A}$
is 
\[
\hocoendlim_{(A,B)\in (\w\aC^{\op})^{2}} N(\WCW(A,B)_{\wco}),
\]
where, as in Theorem~\ref{thmwordcomp}, $\WCW(A,B)_{\wco}$ denotes the
full subcategory of $\WCW(A,B)$ of diagrams where the forward arrow is
a weak cofibration.
Lemma~\ref{lemreversi} then finishes the argument for the case $n=2$.  Now
assume $n\geq 3$.

\begin{lem}
The diagonal of the bisimplicial set $\mathbb{A}$ is weakly equivalent
to the diagonal of the bisimplicial set $\mathbb{B}$ that has as its
$(q,r)$ simplices the commutative diagrams of the following form
\[
\xymatrix@C-1.25ex@R-\redsquish{%
&Y_{0,1}\ar[r]^{wc}\ar[d]_{\htp}
&Z_{0,1}\ar[d]_{\htp}
&Y_{0,2}\ar[l]_{\htp}\ar[r]^{wc}\ar[d]_{\htp}
&\dotsb\ar[d]_{\htp}
&Y_{0,n-1}\ar[l]_{\htp}\ar[r]^{wc}\ar[d]_{\htp}
&Z_{0,n-1}\ar[d]^{\htp}
&B_{0}\ar[l]_-{\htp}\\
&\vdots\ar[r]^{wc}\ar[d]_{\htp}
&\vdots\ar[d]_{\htp}
&\vdots\ar[l]_{\htp}\ar[r]^{wc}\ar[d]_{\htp}
&\dotsb\ar[d]_{\htp}
&\vdots\ar[l]_{\htp}\ar[r]^{wc}\ar[d]_{\htp}
&\vdots\ar[d]^{\htp}\\
A_{0}
&Y_{r,1}\ar[r]^{wc}\ar[l]_{\htp}
&Z_{r,1}
&Y_{r,2}\ar[l]_{\htp}\ar[r]^{wc}
&\dotsb
&Y_{r,n-1}\ar[l]_{\htp}\ar[r]^{wc}
&Z_{r,n-1}
}
\]
together with sequences of weak equivalences 
\[
\xymatrix@-1pc{%
A_{q}&\dotsb \ar[l]_{\htp}&A_{0}\ar[l]_{\htp}
&
&B_{0}&\dotsb\ar[l]_{\htp}&B_{q}\ar[l]_{\htp}.
}
\]
\end{lem}

\begin{proof}
We have an inclusion of $\mathbb{A}$ in $\mathbb{B}$ by inserting
identity maps in the appropriate columns.  The lemma now follows from
Theorem~\ref{thmHCLF} and homotopy calculus of two-sided fractions
\cite[9.4,9.5]{DKHammock}. 
\end{proof}

\begin{lem}
The diagonal of the bisimplicial set $\mathbb{B}$ is weakly equivalent
to the diagonal of the bisimplicial set $\mathbb{C}$ that has as its
$(q,r)$ simplices the commutative diagrams of the following form
\[
\xymatrix@C-1.25ex@R-\redsquish{%
&Y_{0,1}\ar[r]^{wc}\ar[d]_{\htp}&Z_{0,1}\ar[d]_{\htp}
&Y_{0,2}\ar[r]^{wc}\ar[d]_{\htp}&Z_{0,2}\ar[d]_{\htp}
\ar@{{}{}{}}[r]|{\textstyle \dotsb}
&Y_{0,n-1}\ar[r]^{wc}\ar[d]_{\htp}&Z_{0,n-1}\ar[d]^{\htp}
&B_{0}\ar[l]_-{\htp}\\
&\vdots\ar[r]^{wc}\ar[d]_{\htp}&\vdots\ar[d]_{\htp}
&\vdots\ar[d]_{\htp}\ar[r]^{wc}&\vdots\ar[d]_{\htp}
&\vdots\ar[r]^{wc}\ar[d]_{\htp}&\vdots\ar[d]^{\htp}\\
A_{0}
&Y_{r,1}\ar[r]^{wc}\ar[l]_{\htp}&Z_{r,1}
&Y_{r,2}\ar[r]^{wc}\ar[uul]^{\htp}&Z_{r,2}
\ar@{{}{}{}}[r]|{\textstyle \dotsb}
&Y_{r,n-1}\ar[r]^{wc}\ar[uul]|{\textstyle\mathstrut\dotsb\mathstrut}&Z_{r,n-1}
}
\]
together with sequences of weak equivalences 
\[
\xymatrix@-1pc{%
A_{q}&\dotsb \ar[l]_{\htp}&A_{0}\ar[l]_{\htp}
&
&B_{0}&\dotsb\ar[l]_{\htp}&B_{q}\ar[l]_{\htp}.
}
\]
\end{lem}

\begin{proof}
Fix $A_{0}$, $B_{0}$ and consider the simplicial set $\mathbb{C}_{k}$
that has its $r$-simplices the pairs of commutative diagrams
\[
\xymatrix@C-1.25ex@R-\redsquish{%
&Y_{0,1}\ar[r]^{wc}\ar[d]_{\htp}
&Z_{0,1}\ar[d]_{\htp}
&Y_{0,2}\ar[l]_{\htp}\ar[r]^{wc}\ar[d]_{\htp}
&\dotsb\ar[d]_{\htp}
&Y_{0,k}\ar[l]_{\htp}\ar[r]^{wc}\ar[d]_{\htp}
&Z_{0,k}\ar[d]^{\htp}\\
&\vdots\ar[r]^{wc}\ar[d]_{\htp}
&\vdots\ar[d]_{\htp}
&\vdots\ar[l]_{\htp}\ar[r]^{wc}\ar[d]_{\htp}
&\dotsb\ar[d]_{\htp}
&\vdots\ar[l]_{\htp}\ar[r]^{wc}\ar[d]_{\htp}
&\vdots\ar[d]^{\htp}\\
A_{0}
&Y_{r,1}\ar[r]^{wc}\ar[l]_{\htp}
&Z_{r,1}
&Y_{r,2}\ar[l]_{\htp}\ar[r]^{wc}
&\dotsb
&Y_{r,k}\ar[l]_{\htp}\ar[r]^{wc}
&Z_{r,k}
}
\]
and
\[
\xymatrix@C-1.25ex@R-\redsquish{%
Z_{0,k}
&Y_{0,k+1}\ar[r]^{wc}\ar[d]_{\htp}&Z_{0,k+1}\ar[d]_{\htp}
\ar@{{}{}{}}[r]|{\textstyle \dotsb}
&Y_{0,n-1}\ar[r]^{wc}\ar[d]_{\htp}&Z_{0,n-1}\ar[d]^{\htp}
&B_{0}\ar[l]_-{\htp}\\
&\vdots\ar[d]_{\htp}\ar[r]^{wc}&\vdots\ar[d]_{\htp}
&\vdots\ar[r]^{wc}\ar[d]_{\htp}&\vdots\ar[d]^{\htp}\\
&Y_{r,k+1}\ar[r]^{wc}\ar[uul]^{\htp}&Z_{r,k+1}
\ar@{{}{}{}}[r]|{\textstyle \dotsb}
&Y_{r,n-1}\ar[r]^{wc}\ar[uul]|{\textstyle\mathstrut\dotsb\mathstrut}&Z_{r,n-1}
}
\]
together with sequences of weak equivalences 
\[
\xymatrix@-1pc{%
A_{q}&\dotsb \ar[l]_{\htp}&A_{0}\ar[l]_{\htp}
&
&B_{0}&\dotsb\ar[l]_{\htp}&B_{q}\ar[l]_{\htp}.
}
\]
Then $\mathbb{C}_{1}$ is $\mathbb{C}$ and $\mathbb{C}_{n-1}$ is
$\mathbb{B}$.  We construct a zigzag of weak equivalences between
$\mathbb{C}_{k+1}$ and 
$\mathbb{C}_{k}$ with the diagonal of a bisimplicial set in the middle.
Let $\mathbb{D}_{k}$ be the bisimplicial set 
that has as its
$(r,s)$ simplices the commutative diagrams of the following form
\[
\xymatrix@C-1.25ex@R-\redsquish{%
\relax\ar@{<->}[r]&\dotsb\ar[r]\ar[d]_{\htp}
&Z_{0,k}\ar[d]_{\htp}&Y_{0,k+1}\ar[l]_{\htp}\ar[r]^{wc}\ar[d]_{\htp}
&Z_{0,k+1}\ar[d]_{\htp}
&Y_{0,k+2}\ar[r]^{wc}\ar[d]_{\htp}
&Z_{0,k+2}\ar[d]_{\htp}\ar@{{}{}{}}[r]|{\textstyle \dotsb}
&\qquad\\
\relax\ar@{<->}[r]&\dotsb\ar[r]\ar[d]_{\htp}
&\vdots\ar[d]_{\htp}&\vdots\ar[l]_{\htp}\ar[r]\ar[d]_{\htp}
&\vdots\ar[d]_{\htp}
&\vdots\ar[d]_{\htp}\ar[r]^{wc}
&\vdots\ar[d]_{\htp}\ar@{{}{}{}}[r]|{\textstyle \dotsb}
&\qquad\\
\relax\ar@{<->}[r]&\dotsb\ar[r]\ar[d]_{\htp}
&Z_{r,k}\ar[d]_{\htp}&Y_{r,k+1}\ar[l]_{\htp}\ar[r]
&Z_{r,k+1}
&Y_{r,k+2}\ar[uul]^{\htp}\ar[r]^{wc}
&Z_{r,k+2}\ar@{{}{}{}}[r]|{\textstyle \dotsb}
&\qquad\\
\relax\ar@{<->}[r]&\dotsb\ar[r]\ar[d]_{\htp}
&Z'_{0,k}\ar[d]_{\htp}\\
\relax\ar@{<->}[r]&\dotsb\ar[r]\ar[d]_{\htp}
&\vdots\ar[d]_{\htp}\\
\relax\ar@{<->}[r]&\dotsb\ar[r]
&Z'_{s,k},
}
\]
where to the right the columns look like those in $\mathbb{C}$ and to the
left the columns look like those in $\mathbb{B}$.  Regarding
$\mathbb{C}_{k+1}$ as a bisimplicial set constant in the $s$-direction,
we obtain a bisimplicial map $\mathbb{D}_{k}\to \mathbb{C}_{k+1}$ by
forgetting the $Y'_{i,j}$ and $Z'_{i,j}$ parts of the diagram.  It is
easy to see that this map is a weak equivalence using the fact that
the undercategory of
\[
\xymatrix@-1pc{%
A_{0}&Y_{r,k}\ar[l]_{\htp}\ar[r]^{wc}&\dotsb\ar[r]&Z_{r,k}
&Y_{r,k+1}\ar[l]_{\htp}
}
\]
in $\WCdotsW(A_{0},Y_{r,k+1})$ has contractible nerve.  To relate $\mathbb{D}_{k}$ and
$\mathbb{C}_{k}$, we regard $\mathbb{C}_{k}$ as the diagonal of the
bisimplicial set where $Y_{i,j}$, $Z_{i,j}$ are indexed in
$i=0,\dotsc,r$ for $j>k$ and in $i=0,\dotsc,s$ for $j\leq k$; to match
the notation in $\mathbb{D}_{k}$, we will refer to these latter
entries as $Y'_{i,j}$ and $Z'_{i,j}$ (for $j\leq k$).
We then get a bisimplicial map from
$\mathbb{D}_{k}$ to $\mathbb{C}_{k}$ by
forgetting the $Y_{i,j}$ and $Z_{i,j}$ parts of the $\mathbb{D}_{k}$
diagram for $j\leq 
k$, and using the composite map $Y_{r,k+1}\to Z'_{0,k}$.  For fixed
$s$, this map is
a simplicial homotopy equivalence: The inverse equivalence fills in
the $Y_{i,j}$ and $Z_{i,j}$ entries for $j\leq k$ with $Y'_{i,0}$ and
$Z'_{i,0}$.  The composite on $\mathbb{C}_{k}$ is the identity, and
the composite on $\mathbb{D}_{k}$ is homotopic to the identity by the
usual argument. The map on diagonals from $\mathbb{D}_{k}$ to
$\mathbb{C}_{k}$ is then a weak equivalence.
\end{proof}

Finally, to complete the argument, by Lemma~\ref{lemreversi} and
Theorem~\ref{thmwordcomp}, it 
suffices to see that the diagonal of
the bisimplicial set $\mathbb{C}$ is weakly equivalent to
\[
\hocoendlim_{(A,C_{1},\dotsc,C_{n-2},B)\in (\w\aC^{op})^{n}}
\hbox to -3em{\hss}
N(\WCW(C_{n-2},B)_{\wco})\times \dotsb \times N(\WCW(A,C_{1})_{\wco}).
\] 
We can view the latter as the bisimplicial set with $q$-direction the
nerve of $(\w\aC^{op})^{n}$ and $r$-direction the nerve of the
$\WCW(-,-)_{\wco}$.  The $(q,r)$-simplices then look very similar to
the diagrams that define 
$\mathbb{C}$, except that in place of the maps $Z_{0,k}\leftarrow
Y_{r,k+1}$, we have sequences of maps of the form
\[
\xymatrix@-1pc{%
Z_{0,k}
&C_{q,k}\ar[l]_{\htp}&\dotsb \ar[l]_{\htp}&C_{0,k}\ar[l]_{\htp}
&Y_{r,k+1}\ar[l]_{\htp}.
}
\]
Composing induces a bisimplicial map from the homotopy coend to
$\mathbb{C}$ that is easily seen to be a weak equivalence.

Part~(iii) of Theorem~\ref{thmhocoend} follows from
Proposition~\ref{prophocoend} and part~(ii).  Thus, it remains to
prove part~(ii), namely, that $N\w\aC$ is weakly
equivalent to the disjoint union of $B\Laut(X)$.  Fixing $X$ in $\aC$,
the undercategory of $X$ in $\w\aC$ has contractible nerve.  
Then the (cartesian) commutative diagram of categories 
\[
\xymatrix@-1pc{%
\WW(X,X)\ar[r]\ar[d]&\w\aC\bs X\ar[d]\\
\w\aC\bs X\ar[r]&\w\aC
}
\]
satisfies the hypotheses of Theorem~\ref{thmC}.   Thus,
$\Laut(X)\htp\WW(X,X)$ is equivalent to the loop space of $N\w\aC$
based at $X$. 

This completes the proof of Theorem~\ref{thmhocoend}, which implies
Theorems~\ref{thmhocoendone}  and ~\ref{thmhocoendtwo}.  We close the
section with a remark on Theorem~\ref{thmhocoendone}.

\begin{rem}
The decomposition of Theorem~\ref{thmhocoendone}  does not
fit into a simplicial structure 
to give a ``construction'' of the algebraic $K$-theory spectrum.  An
indirect construction of the algebraic $K$-theory spectrum for certain
categories enriched in simplicial sets via the category of simplicial
functors can be found in \cite[\S4]{ToenVezzosi}.  The Dwyer-Kan
simplicial localization of a category that admits functorial
factorization satisfies the hypotheses there, 
by Theorem~\ref{thmhococart}.

A direct construction in terms of the Dwyer-Kan simplicial
localization would include a description of face and 
degeneracy maps fitting the pieces together compatibly with the
simplicial structure on the $\Sdot$ construction.  Although we do not
produce such a construction, we can reinterpret some of the simplicial
structure maps of $\Sdot$ in terms of the spaces described above.
The degeneracy maps, as in $\Sdot$, are induced simply by repeating an
object $X_{i}$ and using the identity map in $L\aC(X_{i},X_{i})$,
where we understand $X_{0}$ as the distinguished zero object of
$\aC$.  The face maps $d_{2},\dotsc,d_{n-1}$ are induced by composition 
\[
L\aC(X_{i},X_{i+1})\times L\aC(X_{i-1},X_{i})
\to L\aC(X_{i-1},X_{i+1}).
\]
The face maps $d_{1}$ and $d_{n}$ essentially drop the first and
$n$-th objects, respectively.  The face map $d_{0}$, which in
$\Sdot[n]$ corresponds to replacing the sequence of cofibrations
$X_{1}\to \dotsb \to X_{n}$ with the quotient $X_{2}/X_{1}\to\dotsb\to
X_{n}/X_{1}$, is the impediment to making the spaces above into a
simplicial object.  Under the
hypotheses of the Theorem~\ref{thmhocoend}, Dwyer-Kan mapping complexes take
pushouts along cofibrations to homotopy pullbacks, i.e., take homotopy
cocartesian squares to homotopy cartesian squares.  Roughly speaking,
the face map $d_{0}$ on the spaces above would involve
composition and taking homotopy fibers.
\end{rem}

\section{Proof of Theorem~\ref{thmhh}}\label{secphh}

In this section we prove Theorem~\ref{thmhh} following ideas in
\cite{MandellHH}.  The first key step is relating the homotopy
categories of undercategories to the higher homotopy data implicit in
the Dwyer-Kan simplicial localization.  

\begin{thm}\label{thmDKunder}
Let $\aC$ be a Waldhausen category that admits a HCLF and let $A$ be an object
of $\aC$.  
Let $\overrightarrow{B}=A\to B$ be a cofibration viewed as an object of
$\aC\bs A$, let $\overrightarrow{C}=A\to C$ be an object of $\aC\bs A$, and
write $\overrightarrow{A}$ for $\id\colon A\to A$, viewed as an object of
$\aC\bs A$.
The following square is homotopy cartesian:
\[
\xymatrix@-1pc{%
L(\aC\bs A)(\overrightarrow{B},\overrightarrow{C})\ar[r]\ar[d]
&L(\aC\bs A)(\overrightarrow{A},\overrightarrow{C})\ar[d]\\
L\aC(B,C)\ar[r]
&L\aC(A,C)
}
\]
\end{thm}

\begin{proof}
Since $\aC$ admits a HCLF, it suffices to show that the square
\[
\xymatrix@-1pc{%
\WCA(\overrightarrow{B},\overrightarrow{C})\ar[r]\ar[d]
&\WCA(\overrightarrow{A},\overrightarrow{C})\ar[d]\\
\WC(B,C)\ar[r]&\WC(A,C)
}
\]
is homotopy cartesian, where we have written $\WCA$ for the categories
of words $\WC$ in $\aC\bs A$ to avoid confusion.  We apply
Theorem~\ref{thmC}.  An easy check of the
definitions shows that this square 
satisfies the hypothesis~(ii), and it satisfies hypothesis~(i) by 
Lemma~\ref{lemthmhococart} (with $E=C$).
\end{proof}

The previous theorem identifies the Dwyer-Kan function complexes in
$\aC\bs A$ for cofibrations.  In the context of Theorem~\ref{thmhh},
factorization allows us to extend this to compute the
Dwyer-Kan function complexes for arbitrary objects of $\aC \bs A$.
The following proposition is immediately clear when $\aC$ admits
functorial factorization; the non-functorial case is treated in
Appendix~\ref{appmain}.

\begin{prop}\label{propunderiscof}
Let $\COF{A}$ denote the full subcategory of $\aC\bs A$ consisting of
the cofibrations.  If $\aC$ admits factorization, then the
inclusion $\COF{A}\to \aC \bs A$ is a DK-equivalence.
\end{prop}

For a map $f\colon B \to C$ in $\aC$, let $\Omega_{f}L\aC(B,C)$ denote
the space of based loops in the geometric realization of $LC(B,C)$,
based at the vertex $f$.  Thinking of $\Omega_{f}L\aC(B,C)$ as the
homotopy pullback of the diagram
\[
\xymatrix{%
&L\aC(B,C)\ar[d]^{\Delta}\\
\{(f,f)\}\ar[r]&L\aC(B,C)\times L\aC(B,C),
}
\]
then up to weak equivalence, we can identify the lower right term as
$L\aC(B\vee B,C)$ by 
Theorem~\ref{thmhococart}; we then get the following result as a
corollary of the previous proposition and theorem.

\begin{cor}\label{corloop}
Let $\aC$ be a Waldhausen category that admits factorization.
Let $f\colon B\to C$ be a map in $\aC$.  Viewing $B$ as an object
under $B\vee B$ via the codiagonal map, and $C$ as such via the composite with
$f$, then the loop space $\Omega_{f}L\aC(B,C)$ is weakly equivalent to
$L(\aC\bs B\vee B)(B,C)$.
\end{cor}

In comparing function complexes for $\aC$ to function complexes for
$\aD$, we need the following proposition, which is an easy
consequence of Theorem~\ref{thmDKunder} and 
Proposition~\ref{propunderiscof}.

\begin{prop}\label{propunderwewe}
Let $A'\to A$ be a map in $\aC$ that is an isomorphism in $\Ho\aC$.
If $\aC$ admits factorization,
then the induced functor $\aC\bs A\to \aC\bs A'$ is a DK-equivalence.
\end{prop}

Finally, we prove Theorem~\ref{thmhh}.  Clearly, by
Theorem~\ref{thmDKunder} and Proposition~\ref{propunderiscof}, a
DK-equivalence implies an equivalence of homotopy categories and
homotopy categories of all under categories.  For the converse, note
that $\COF{(B \vee B)}$ inherits from $\aC$ the property of
admitting factorization.  Likewise, note
that for an arbitrary component of $L\aC(B,C)$, we can find a vertex
$\phi$ of the form
\[
\xymatrix@-1pc{%
B\ar[r]^{f}&X&C\ar[l]_{\htp}
}
\]
by HCLF.   The loop space based at $\phi$,
$\Omega_{\phi}L\aC(B,C)$, is 
then homotopy equivalent to
$\Omega_{f}L\aC(B,X)$, and so weakly equivalent to $L(\aC\bs (B\vee
B))(B,X)$, as per Corollary~\ref{corloop}. Thus, iterating
Corollary~\ref{corloop} identifies the homotopy groups of $L\aC(B,C)$
at arbitrary basepoints in terms of sets of maps in the homotopy
categories of undercategories, as in \cite[5.4]{MandellHH}.

Specifically, we can identify $\pi_{n}(L\aC(B,C))$ based at $\phi$ as
$\Ho(\aC\bs S^{n-1})(B,X)$, for certain objects $S^{n-1}$  formed
inductively as follows: Starting with $S^{-1}=*$ and $B^{0}=B$,
$S^{n}$ is formed as 
the coproduct $B^{n}\cup_{S^{n-1}}B^{n}$ in $\COF{S^{n-1}}$ where $B^{n}$
is an object of $\COF{S^{n-1}}$ with a weak equivalence $B^{n}\to B$
in $\aC\bs S^{n-1}$.  Now in
$\aD$, we can perform the analogous construction starting with
$S^{-1}_{\aD}=*$ and $B^{0}_{\aD}=FB$ to form $S^{n-1}_{\aD}$ and
$B^{n}_{\aD}$.  When inductively we choose the weak equivalence
$B^{n}_{\aD}\to FB$ to factor through $FB^{n}$ in $\aD\bs
S^{n-1}_{\aD}$, then $S^{n}_{\aD}\to FB$ factors through $FS^{n}$.
Since we have not assumed that $F$ is weakly exact, we cannot conclude
that the map $S^{n}_{\aD}\to FS^{n}$ is a weak equivalence; however,
we do have the following lemma, which then completes the proof of
Theorem~\ref{thmhh}. 

\begin{lem}
For all $n$, the restriction of $F$ to a functor $\aC\bs S^{n}\to
\aD\bs S^{n}_{\aD}$ induces an equivalence of homotopy categories
$\Ho(\aC\bs S^{n})\to \Ho(\aD\bs S^{n}_{\aD})$.
\end{lem}

\begin{proof}
For $n=-1$, $S^{-1}=*$ and $S^{-1}_{\aD}=*$ are the initial object in
each category; the equivalence $\Ho\aC\to \Ho\aC$ is part of the
hypothesis on $F$.  Now by induction, assume that $\Ho(\aC\bs
S^{n-1})\to \Ho(\aD\bs S^{n-1}_{\aD})$ is an equivalence.  By
Theorem~\ref{thmhococart}, we have that
$S^{n}=B^{n}\cup_{S^{n-1}}B^{n}$ is the coproduct of two copies of $B$
in $\Ho(\aC\bs S^{n-1})$ and
$S^{n}_{\aD}=B^{n}_{\aD}\cup_{S^{n-1}_{\aD}}B^{n}_{\aD}$ is the
coproduct of two copies of $FB$ in 
$\Ho(\aD\bs S^{n-1}_{\aD})$.  It follows that the map $S^{n}_{\aD}\to
FS^{n}$ is an isomorphism in $\Ho(\aD\bs S^{n-1}_{\aD})$ and hence in
$\Ho\aD$.  By Proposition~\ref{propunderwewe}, the map $S^{n}_{\aD}\to
FS^{n}$ induces an equivalence $\Ho(\aD\bs FS^{n})\to \Ho(\aD\bs
S^{n}_{\aD})$, and by hypothesis on $F$, the functor $\Ho(\aC\bs
S^{n})\to \Ho(\aD\bs FS^{n})$ is an equivalence.  The functor 
$\Ho(\aC\bs S^{n})\to \Ho(\aD\bs S^{n}_{\aD})$ in question is the
composite of these two equivalences.
\end{proof}

\section{Proof of Theorem~\ref{propapprox}}\label{secpwaldapp}

This section proves Theorem~\ref{propapprox}, which relates the
approximation property to the homotopy categories of the
undercategories.  It is convenient to prove the theorem in the
following form.

\begin{thm}\label{thmpropapprox}
Let $\aC$ be a Waldhausen category where every map factors as a
cofibration followed by a weak equivalence. Let $\aD$ be a category with
weak equivalences and let $F\colon \aC\to \aD$ be a functor that
satisfies the approximation property, preserves finite coproducts, and
preserves pushouts where one leg in $\aC$ is a cofibration.  Then $F$
induces an equivalence $\Ho\aC\to \Ho\aD$.
\end{thm}

Note that we do not assume that $\aD$ has all finite coproducts or pushouts;
only the finite coproducts and pushouts required by the hypotheses are
assumed to exist.

Theorem~\ref{thmpropapprox}  implies Theorem~\ref{propapprox}:  When
$\aC$ has the property that every map can be factored as a cofibration
followed by a weak equivalence, then for any object $A$, the inclusion
of $\COF{A}$ in $\aC\bs A$ (with notation as in Proposition~\ref{propunderiscof}
above) satisfies the hypothesis of Theorem~\ref{thmpropapprox} and so
induces an equivalence of homotopy categories.  Likewise, $F$ regarded
as a functor $\COF{A}\to \aD\bs FA$ satisfies the hypothesis of
Theorem~\ref{thmpropapprox} and so induces an equivalence of homotopy
categories.  It follows that $F$ induces an equivalence of homotopy
categories $\Ho(\aC\bs A)\to \Ho(\aD\bs FA)$. 

The remainder of the section is devoted to the proof of
Theorem~\ref{thmpropapprox}. 
We begin by constructing a functor $R\colon \aD\to \Ho\aC$.  For each
object $X$ in $\aD$, apply the approximation property to the initial
map $F*\to X$ to choose an object $RX$ in $\aC$ and a weak equivalence
$\epsilon_{X}\colon FRX\to X$ in $\aD$.  For each map in $\aD$, $f\colon X\to Y$, apply
the approximation property to the map
\[
F(RX \vee RY) \iso FRX \vee FRY \xrightarrow{f\circ \epsilon_{X}+\epsilon_{Y}} Y
\]
to obtain an object $Qf$ of $\aC$, cofibrations $RX\to Qf$ and $RY\to Qf$, 
and a commuting diagram in $\aD$.
\begin{equation}\label{diagQf}
\begin{gathered}
\xymatrix{%
FRX\ar[r]\ar[d]_{\htp}^{\epsilon_{X}}&FQf\ar[dr]_{\htp}
&FRY\ar[l]\ar[d]_{\htp}^{\epsilon_{Y}}\\
X\ar[rr]_{f}&&Y
}
\end{gathered}
\end{equation}
It follows from part~(i) of the approximation property then that the
map $RY\to Qf$ is a weak equivalence, and so we obtain a zigzag in $\aC$,
\[
\xymatrix@C-1pc{%
RX\ar[r]&Qf&RY.\ar[l]_{\htp}
}
\]
Let $Rf$ be the map $RX\to RY$ in $\Ho\aC$ represented by this zigzag.
The following lemma implies that $Rf$ is independent of the choice of $Qf$.

\begin{lem}\label{lemeqQf}
Let $f\colon X\to Y$ be a map in $\aD$.  Let 
\[
\xymatrix@-1pc{%
RX\ar[r]^{\alpha}&B&RY\ar[l]^{\htp}_{\beta}
}
\]
be any zigzag in $\aC$.  If there exists a
map $\gamma \colon FB\to Y$
that makes the diagram 
\[
\xymatrix{%
FRX\ar[r]\ar[d]_{\htp}^{\epsilon_{X}}&FB\ar[dr]_{\gamma}
&FRY\ar[l]_{\htp}\ar[d]_{\htp}^{\epsilon_{Y}}\\
X\ar[rr]_{f}&&Y
}
\]
commute in $\aD$, then $Rf=\beta^{-1}\circ \alpha$ in $\Ho\aC$.
\end{lem}

\begin{proof}
By construction, the map $RX\vee RY\to Qf$ is a cofibration, and
by factorization, we can assume without loss of
generality that the map $\alpha + \beta \colon RX\vee RY\to B$ is a
cofibration.  Since the maps $FB\to Y$ and $FQf\to Y$ both compose
to the same maps $FRX\to Y$ and the same maps $FRY\to Y$, we obtain
a map 
\[
F(B\cup_{RX\vee RY}Qf) \iso FB\cup_{FRX\vee FRY}FQf\to Y.
\]
Applying the approximation property to this map, we obtain an object
$C$ in $\aC$ and a map 
\[B \cup_{RX\vee RY} Qf \to C\]
in $\aC$ such that the composites $B\to C$ and $Qf\to C$ are both weak
equivalences and restrict to the same maps $RX\to C$ and the same maps
$RY\to C$.  Thus, we have the following commutative diagrams in $\aC$.
\[\xymatrix@-1.5pc{%
&RX\ar[dl]\ar[dr]&&&&RY\ar[dl]_{\htp}\ar[dr]^{\htp}\\
B\ar[dr]_{\htp}&&Qf\ar[dl]^{\htp}
&&B\ar[dr]_{\htp}&&Qf\ar[dl]^{\htp}\\
&C&&&&C
}
\]
We see from these diagrams that the maps in $\Ho\aC$  represented by the zigzags
\[
\xymatrix@-1pc{%
RX\ar[r]&Qf&RY,\ar[l]_{\htp}&&
RX\ar[r]&C&RY,\ar[l]_{\htp}&&
RX\ar[r]&B&RY,\ar[l]_{\htp}
}
\]
coincide.  The first is $Rf$ and the third is $\beta^{-1}\circ \alpha$.
\end{proof}

\begin{thm}
$R$ is a functor $\aD\to \Ho\aC$.
\end{thm}

\begin{proof}
Applying Lemma~\ref{lemeqQf} with $\alpha=\id_{RX}=\beta$ and $\gamma
=\epsilon_{X}$, it follows that $R\id_{X}$ is $\id_{RX}$.  Now given
maps $f\colon X\to Y$ and $g\colon Y\to Z$ in $\aD$, let
$B=Qf\cup_{RY}Qg$.  Then we see from the commuting diagram on the left
\[
\xymatrix@-1.5pc{%
RX\ar[dr]&&RY\ar[dl]_{\htp}\ar[dr]&&RZ\ar[dl]_{\htp}\\
&Qf\ar[dr]&&Qg\ar[dl]^{\htp}&&&&RX\ar[rr]&&B&&RZ\ar[ll]_{\htp}\\
&&B
}
\]
that $Rg\circ Rf$ is represented by the zigzag on the right.
Applying Lemma~\ref{lemeqQf} to $g \circ f$ with
$\alpha$ and $\beta$ the maps $RX\to B$ and $RZ\to B$ above
and $\gamma\colon FB\to Z$ the map induced by the maps
$FQf\to Y\to Z$ and $FQg\to Z$, we see that $R(f\circ
g)=\beta^{-1}\circ \alpha =Rf\circ Rg$.
\end{proof}

Clearly $R$ takes weak equivalences in $\aD$ to isomorphisms in
$\Ho\aC$, and so $R$ factors through a functor $\Ho\aD\to \Ho\aC$ that we
also denote as $R$.  It is clear from Diagram~\ref{diagQf} that
$\epsilon$ is a natural isomorphism from $FR$ to the identity in $\Ho\aD$.
For $C$ an object of $\aC$, applying the approximation property to the map
\[
F(C\vee RFC)\iso FC\vee FRFC\to FC
\]
constructs an object $PC$ in $\aC$ with weak equivalences $C\to PC$
and $RFC\to PC$.  This then gives a zigzag in $\aC$ that represents an
isomorphism in $\Ho\aC$ from $C$ to $RFC$.  It is straightforward to
verify using Lemma~\ref{lemeqQf} that this
isomorphism is natural.


\appendix
\numberwithin{equation}{subsection}

\section{{USE} of factorizations and mapping cylinders}
\label{appbm}

In this appendix, we introduce the concept of a universal simplicial
equivalence (USE) of a space of factorizations or a space of mapping
cylinders for a Waldhausen category $\aC$.  A
USE of factorizations or mapping cylinders is a way of encoding the
data of a contractible space of choices of factorization or mapping
cylinder for each morphism in $\aC$ (or a distinguished subcategory).
Although our formulation is new (and relies on the definition of a
universal quasifibration), this kind of approach to handling a lack of
functoriality in factorizations derives from the work of Dwyer and
Kan in \cite{DKModel}.

The purpose of this appendix is to generalize the results
of \cite{BlumbergMandell} on homotopy cocartesian squares and on the
$\Spdot$ construction from requiring functorial factorization of weak
cofibrations to requiring a USE of factorizations of weak
cofibrations.  In addition, we show that the existence of
(non-functorial) factorizations implies a USE of mapping cylinders.  
In the next appendix, we show how to modify the arguments of the body
of the paper to remove functoriality hypotheses assuming just
existence of factorization or a USE of mapping cylinders for weak
cofibrations.

Throughout this section $\aC$ denotes a Waldhausen category whose weak
equivalences satisfy the two out of three property (i.e., are
saturated in the sense of Waldhausen). 

\subsection{Categories of factorizations and mapping cylinders}

We begin with the formal
definition of the categories of factorizations and mapping 
cylinders for $\aC$.

\begin{defn}
The category $\Fac{\aC}$ is the
full subcategory of diagrams
\[
\xymatrix@-1pc{
A \ar@{ >->}[r] & X \ar[r]^{\htp} & B
}
\]
where the map $A\to X$ is a cofibration and the map $X\to B$ is a weak
equivalence.  The forgetful functor $\Fac{\aC}\to \Ar\aC$ sends the
diagram to the composite map $A\to B$.

The category $\MC{\aC}$ is the full subcategory of diagrams
\[
\xymatrix@R-1pc{%
&B\ar[d]\ar[dr]^{=}\\
A\ar@{ >->}[r]&X\ar[r]_{\htp}&B,
}
\]
where the map $A\to X$ is a cofibration, the map $X\to B$ is a weak
equivalence, and the composite map $B\to B$ the identity.  The
forgetful functor $\MC{\aC}\to\Fac{\aC}$ forgets the map $B\to X$, and
we obtain a composite forgetful functor $\MC{\aC}\to\Ar\aC$.
\end{defn}

In this terminology, the existence of factorizations is equivalent to
the forgetful functor $\Fac{\aC}\to\Ar\aC$ being surjective on
objects.  Likewise, functorial factorization as defined in
Section~\ref{secweak} consists of a functor
$T\colon \Ar\aC\to \Fac{\aC}$ such that the composite with the
forgetful functor $\Fac{\aC}\to \Ar\aC$ is the identity on $\Ar\aC$.
We write $\wcAr\aC$ for the full
subcategory of $\Ar\aC$ consisting of the weak cofibrations.
Functorial factorization of weak cofibrations or functorial
mapping cylinders for weak cofibrations then consists of a functor $T\colon
\wcAr\aC\to \Fac{\aC}$ or $T\colon \wcAr\aC\to \MC{\aC}$ such that the
composite with the forgetful functor $\Fac\aC\to \Ar\aC$ or
$\MC{\aC}\to \Ar\aC$ is the inclusion $\wcAr\aC\to \Ar\aC$.

\subsection{Universal simplicial equivalences}

Building on the terminology of Definition~\ref{defunivsq}, we call a
map of simplicial sets $X\subdot\to Y\subdot$ a \term{universal
simplicial equivalence} when it is a universal simplicial
quasifibration and a weak equivalence.  Such a map is characterized by
the property that for any simplicial map $Z\subdot\to Y\subdot$, the
geometric realization of the pullback map $|Z\subdot \times_{Y\subdot}
X\subdot|\to |Z\subdot|$ has contractible fibers.  For a subcategory
$\aS$ of $\Ar\aC$ (such as $\wcAr\aC$), we define a USE of
factorizations and a USE of mapping cylinders as follows.

\begin{defn}\label{defuse}
Let $\aS$ be a subcategory of $\Ar\aC$.  A \term{universal simplicial
equivalence (USE) of factorizations for $\aS$} consists of a
simplicial set $T\subdot$ and a simplicial map $T\subdot\to
N\Fac{\aC}$ such that the composite $T\subdot\to N\Ar\aC$ restricts to
a universal simplicial equivalence 
\[
T\subdot\times_{N\Ar\aC}N\aS\to N\aS.
\]
A \term{universal simplicial
equivalence (USE) of mapping cylinders for $\aS$} consists of a
simplicial set $T\subdot$ and a simplicial map $T\subdot\to
N\Fac{\aC}$ such that the composite $T\subdot\to N\Ar\aC$ restricts to
a universal simplicial equivalence 
\[
T\subdot\times_{N\Ar\aC}N\aS\to N\aS.
\]
When $\aS$ is the category of weak cofibrations $\wcAr$, we say that
$T\subdot$ is a USE of factorizations or mapping cylinders for weak
cofibrations. 
\end{defn}

This definition has several immediate consequences.

\begin{prop}Let $\aS$ be a subcategory of $\Ar\aC$.
\begin{enumerate}
\item Functorial factorization or functorial mapping cylinders for
$\aS$ implies a USE of factorizations or mapping cylinders,
respectively, with $T\subdot=N\aS$ and with $NT$ as the map $T\subdot$ to
$N\Fac{\aC}$ or $N\MC{\aC}$.
\item A USE of factorizations or mapping cylinders for $\aS$ implies a
USE of factorizations or mapping cylinders, respectively, for any
subcategory of $\aS$.
\item A USE of mapping cylinders for $\aS$ implies a USE of
factorizations for $\aS$.
\end{enumerate}
\end{prop}

In particular, a USE of factorizations or mapping cylinders for
$\Ar\aC$ implies a USE of factorizations or mapping cylinders,
respectively, for every subcategory of $\Ar\aC$.  Moreover, a USE of
factorizations or mapping cylinders for $\Ar\aC$ implies existence of
(non-functorial) factorizations or mapping cylinders, respectively.
In fact, the following lemma, proved at the end of this appendix, shows
that a USE of mapping cylinders for all arrows, a USE of factorizations
for all arrows, and existence of factorizations for all arrows are all
equivalent.

\begin{lem}\label{lemfactuse}
If $\aC$ admits factorization (not necessarily functorially), then
there is a USE of mapping cylinders for $\Ar\aC$.
\end{lem}

\subsection{Weak cofibrations and homotopy cocartesian squares}

We now generalize the results of Section~2 of \cite{BlumbergMandell}
on homotopy cocartesian diagrams.  In these results, we follow the
notation of Definition~\ref{defuse} and denote the domain of a USE as
$T\subdot$.  For Proposition~2.3 of~\cite{BlumbergMandell}, recall that a full
subcategory $\aB$ of $\aC$ is called a \term{Waldhausen subcategory}
when it forms a Waldhausen category with weak equivalences the weak
equivalences of $\aC$ and with cofibrations the cofibrations $A\to B$
in $\aC$ (between objects $A$ and $B$ of $\aB$) whose quotient
$B/A=B\cup_{A}*$ is in $\aB$.  We say that the Waldhausen subcategory
$\aB$ is \term{closed} if every object of $\aC$ weakly equivalent to
an object of $\aB$ is an object of $\aB$.

\begin{prop}[{\cite[2.3]{BlumbergMandell}}]\label{propwcofpf}
If $\aB$ is a closed Waldhausen subcategory of a Waldhausen category
$\aC$ with a USE of factorizations for weak cofibrations, then $\aB$
has a USE of factorizations for weak cofibrations.  Moreover, a
weak cofibration $f \colon A \to B$ in $\aC$ between objects in
$\aB$ is a weak cofibration in $\aB$ if and only if there exists some
factorization $A \to X\to B$ in the image of $T\subdot$ such that the
cofibration (in $\aC$) $A\to X$ is a cofibration in $\aB$.
\end{prop}

\begin{proof}
Let $f\colon A\to B$ be a weak cofibration in $\aC$ between objects
in $\aB$.  Then $f$ is weakly equivalent by a zigzag in $\aB$ to a
cofibration $f'\colon A'\to B'$ 
in $\aC$,
\[
\xymatrix{%
A'\ar@{ >->}[d]_{f'}\ar[r]^{\htp}&A_{1}\ar[d]_{f_{1}}&\dotsb\ar[l]_{\htp}&\ar[l]_{\htp}
A_{n}\ar[r]^{\htp}\ar[d]_{f_{n}}&A\ar[d]_{f}\\
B'\ar[r]_{\htp}&B_{1}&\dotsb\ar[l]^{\htp}
&\ar[l]^{\htp}B_{n}\ar[r]_{\htp}&B.
}
\]
This diagram specifies a generalized simplicial path in $N\wcAr\aC$
between the vertices $A\to B$ and $A'\to B'$.  It follows that there
exists a generalized simplicial path between a lift of $A\to B$ to
$T_{0}$ and some lift of $A'\to B'$ to $T_{0}$.  Without loss of
generality (replacing the original generalized simplicial path if
necessary), the image of this path in $N\Fac\aC$ is a commutative 
diagram
\[
\xymatrix{%
A'\ar@{ >->}[d]_{f'}\ar[r]^{\htp}&A_{1}\ar@{ >->}[d]
&\dotsb\ar[l]_{\htp}&\ar[l]_{\htp}
A_{n}\ar[r]^{\htp}\ar@{ >->}[d]&A\ar@{ >->}[d]\\
B'\ar[d]_{=}\ar[r]^{\htp}&X_{1}\ar[d]
&\dotsb\ar[l]_{\htp}&\ar[l]_{\htp}
X_{n}\ar[r]^{\htp}\ar[d]_{f_{n}}&X\ar[d]\\
B'\ar[r]_{\htp}&B_{1}&\dotsb\ar[l]^{\htp}
&\ar[l]^{\htp}B_{n}\ar[r]_{\htp}&B.
}
\]
We then get weak equivalences,
\[ 
\xymatrix@C=2em{%
B'/A'\ar[r]^{\htp}
&X_{1}/A_{1}&\dotsb\ar[l]_(.4){\htp}
&\ar[l]_{\htp} X_{n}/A_{n}\ar[r]^{\htp}&X/A,
}
\]
which imply that the map $A\to X$ is a cofibration in $\aB$ if and
only if $f\colon A\to B$ is a weak cofibration in $\aB$.
\end{proof}

The proof of Proposition~2.4 in \cite{BlumbergMandell} does not
actually use the functoriality of the factorizations of weak
cofibrations but just their existence.  It therefore admits the
following generalization.

\begin{prop}[{\cite[2.4]{BlumbergMandell}}]
Let $\aC$ be a Waldhausen category that admits factorization of weak
cofibrations.  If $f \colon A \to B$ 
and $g \colon B \to C$ are weak cofibrations in $\aC$, then $g \circ
f \colon  A \to C$ is a weak cofibration in $\aC$.
\end{prop}

The following proposition generalizes Proposition~2.5 in
\cite{BlumbergMandell}.  The proof is similar to that of
Proposition~\ref{propwcofpf}. 

\begin{prop}[{\cite[2.5]{BlumbergMandell}}]\label{prophococart}
Let $\aC$ be a Waldhausen category with a USE of
factorizations for weak cofibrations.  A commutative diagram
\[
\xymatrix@-1pc{%
A\ar[r]^{f}\ar[d]&B\ar[d]\\
C\ar[r]&D
}
\]
with $f$ a weak cofibration is homotopy cocartesian if and only
if the map $X\cup_{A}C\to D$ is a weak equivalence for
some factorization $A \to X\to B$ in the image of $T\subdot$.
\end{prop}

The previous proposition then implies the following proposition, which
generalizes Proposition~2.6 of \cite{BlumbergMandell}.

\begin{prop}[{\cite[2.6]{BlumbergMandell}}]\label{proppush}
Let $\aC$ be a Waldhausen category with a USE of
factorizations for weak cofibrations. 
\begin{enumerate}
\item Given a commutative cube
\[
\xymatrix@-1.5pc{%
A'\ar[rr]\ar[dd]\ar[dr]&&B'\ar[dd]\ar[dr]\\
\relax&A\ar'[r][rr]\ar'[d][dd]&&B\ar[dd]\\
C'\ar[rr]\ar[dr]&&D'\ar[dr]\\
\relax&C\ar[rr]&&D
}
\]
with the $(A,B,C,D)$-face and $(A',B',C',D')$-face homotopy
cocartesian, if the maps $A'\to A$, 
$B'\to B$, and $C'\to C$ are weak equivalences, 
then the map $D'\to D$ is a weak equivalence.
\item Given a commutative diagram
\[
\xymatrix@-1pc{%
A\ar[r]\ar[d]&B\ar[d]\ar[r]&X\ar[d]\\ 
C\ar[r]&D\ar[r] &Y \\
}
\]
with the square $(A,B,C,D)$ homotopy cocartesian, if either $A\to C$
is a weak cofibration or both $A\to B$ and $B\to X$ are weak
cofibrations, then the $(A,X,C,Y)$
square is homotopy cocartesian if and only if the $(B,X,D,Y)$ square
is homotopy cocartesian.
\end{enumerate}
\end{prop}

\subsection{Comparing $\Sdot$ and $\Spdot$}

Next we compare $\Sdot$ and $\Spdot$ constructions.  As a first step,
we prove the following version of Theorems~2.8 and~2.9 of
\cite{BlumbergMandell} for $F\subdot$ and $F'\subdot$. Recall that
$F_{n}$ and $F'_{n}$ are the Waldhausen categories with objects the sequences of
$n$ composable cofibrations and weak cofibrations respectively.

\begin{thm}\label{thmfpdot}
Let $\aC$ be a Waldhausen category with a USE of
factorizations for weak cofibrations. 
\begin{enumerate}
\item The Waldhausen category $F'_{n}\aC$ has a USE of 
factorizations for weak cofibrations. 
\item The forgetful functor 
$\w F_{n}\aC\to \w F'_{n}\aC$ induces a weak equivalence on nerves.
\end{enumerate}
\end{thm}

\begin{proof}
The arguments are straightforward modifications of the usual pushout
arguments.  As above, let
$T\subdot$ be the domain of the USE of factorizations of 
weak cofibrations in $\aC$. For part~(i), we denote a typical object
of $\Ar F'_{n}$ as
\[
\xymatrix{%
A_{0}\ar[r]\ar[d]&A_{1}\ar[r]\ar[d]&\dotsb \ar[r]&A_{n}\ar[d]\\
B_{0}\ar[r]&B_{1}\ar[r]&\dotsb \ar[r]&B_{n}
}
\]
(where the maps $A_{i}\to A_{i+1}$ and $B_{i}\to B_{i+1}$ are weak
cofibrations).  Such an object is in $\wcAr F'_{n}$ when the maps
$A_{i}\to B_{i}$ are weak cofibrations and for any factorization
$A_{i}\to X\to B_{i}$ in the image of $T\subdot$, the map
$A_{i+1}\cup_{A_{i}}X\to B_{i+1}$ is a weak cofibration.
For $F'_{n}$, let $T^{0}\subdot$ be the
pullback
\[
\xymatrix{%
T^{0}\subdot\ar[r]\ar@{->>}[d]_{\htp}
&T\subdot\ar@{->>}[d]^{\htp}\\
N\subdot\wcAr F'_{n}\ar[r]&N\subdot\wcAr\aC
}
\]
where the bottom map is induced by the zeroth object
functor $F'_{n}\to \aC$.  Using the map $T\subdot\to N\Fac\aC$, we get
a simplicial map from $T^{0}$ to the nerve of the category of diagrams
of the form
\[
\xymatrix@R-1pc{%
A_{0}\ar[r]\ar@{ >->}[d]&A_{1}\ar[r]\ar[dd]&\dotsb \ar[r]&A_{n}\ar[dd]\\
X_{0}\ar[d]_{\htp}\\
B_{0}\ar[r]&B_{1}\ar[r]&\dotsb \ar[r]&B_{n}.
}
\]
We have a functor from this category to $\Ar\aC$ taking the object
pictured above to the pushout map
\[
X_{0}\cup_{A_{0}}A_{1}\to B_{1},
\]
and the composite map $T^{0}\subdot\to \Ar\aC$ factors as a map 
$p_{0}\colon T^{0}\subdot\to \wcAr\aC$.  Inductively, having
constructed $T^{i}\subdot$ and $p_{i}\colon T^{i}\subdot\to \wcAr\aC$,
define $T^{i+1}\subdot$ as the pullback 
\[
\xymatrix{%
T^{i+1}\subdot\ar[r]\ar@{->>}[d]_{\htp}
&T\subdot\ar@{->>}[d]^{\htp}\\
T^{i}\subdot\ar[r]_{p_{i}}&N\subdot\wcAr\aC
}
\]
and $p_{i+1}\colon T^{i+1}\subdot\to \wcAr\aC$ by the analogous
pushout.  By construction, $T^{n}\subdot$ admits a map to the nerve of
the category of diagrams of the form
\[
\xymatrix@R-1pc{%
A_{0}\ar[r]\ar@{ >->}[d]&A_{1}\ar[r]\ar@{ >->}[d]&\dotsb \ar[r]&A_{n}\ar@{ >->}[d]\\
X_{0}\ar[d]_{\htp}\ar[r]&X_{1}\ar[d]_{\htp}\ar[r]&\dotsb\ar[r]&X_{n}\ar[d]^{\htp}\\
B_{0}\ar[r]&B_{1}\ar[r]&\dotsb \ar[r]&B_{n}
}
\]
such that each map $X_{i}\cup_{A_{i}}A_{i+1}\to X_{i+1}$ is a weak
cofibration, i.e., the category $\Fac F'_{n}$.  The composite functor
$T^{n}\subdot\to \wcAr F'_{n}$ is the composite 
\[
T^{n}\subdot\to T^{n-1}\subdot\to \dotsb \to T^{0}\subdot\to \wcAr F'_{n}
\]
and so is a universal simplicial equivalence.  This constructs the USE
of factorizations of weak cofibrations for $F'_{n}$ and proves part~(i).

For part~(ii), let $U^{1}$ be the
pullback
\[
\xymatrix{%
U^{1}\subdot\ar[r]\ar@{->>}[d]_{\htp}
&T\subdot\ar@{->>}[d]^{\htp}\\
N\w F'_{n}\aC\ar[r]&N\wcAr\aC
}
\]
where the bottom map is induced by the functor $f_{1}\colon
\w F'_{n}\aC\to \w\wcAr\aC$ that takes the object
\[
\xymatrix{%
A_{0}\ar[r]^{f_{1}}&A_{1}\ar[r]^{f_{2}}&\dotsb \ar[r]^{f_{n}}&A_{n},
}
\]
to the arrow $f_{1}$.  Using the map $T\subdot\to N\Fac\aC$, we get
a simplicial map from $U^{1}$ to the nerve of (the subcategory of
weak equivalences in) the category of commuting diagrams
of the form
\[
\xymatrix{%
A_{0}\ar@{ >->}[r]^{f'_{1}}\ar[dr]_{f_{1}}&A'_{1}\ar[d]_{\htp}^{\phi_{1}}\ar[dr]^{g_{2}}\\
&A_{1}\ar[r]_{f_{2}}&A_{2}\ar[r]&\dotsb\ar[r]_{f_{n}}&A_{n}
}
\]
where the map $f'_{1}$ is a cofibration, the maps $f_{i}$ are weak
cofibrations, and the map $\phi_{1}$ is a weak equivalence.
Inductively constructing $U^{i+1}$ as the pullback
\[
\xymatrix{%
U^{i+1}\subdot\ar[r]\ar@{->>}[d]_{\htp}
&T\subdot\ar@{->>}[d]^{\htp}\\
U^{i}\ar[r]_{g_{i+1}}&N\wcAr\aC,
}
\]
we obtain a universal simplicial equivalence $U^{n}\to N\w F'_{n}\aC$
and a simplicial map from $U^{n}$ to the nerve of the category of
commuting diagrams of the form
\[
\xymatrix{%
A_{0}\ar[dr]_{f_{1}}\ar@{ >->}[r]^{f'_{1}}&A'_{1}\ar[d]_{\htp}^{\phi_{1}}\ar@{ >->}[r]
&\dotsb \ar@{ >->}[r]&A'_{n}\ar[d]_{\htp}^{\phi_{n}}\\
&A_{1}\ar[r]&\dotsb \ar[r]&A_{n}.
}
\]
This constructs a map $U^{n}\to N\w F_{n}\aC$. 
The maps $\phi_{i}$ in
the diagram above induce a homotopy between the composite map
$U^{n}\to N\w F_{n}\aC\to N\w F'_{n}\aC$ and the universal simplicial
equivalence $U^{n}\to N\w F'_{n}\aC$.  It follows that the right-hand
triangle in the diagram
\[
\xymatrix{%
U^{n}\times_{N\w F'_{n}\aC}N\w F_{n}\aC\ar@{->>}[d]_{\htp}\ar[r]&
U^{n}\ar@{->>}[d]_{\htp}\ar@{.>}[dl]\\
N\w F_{n}\aC\ar[r]&N\w F'_{n}\aC
}
\]
commutes up to homotopy.  Likewise, the maps $\phi_{i}$ induce a
homotopy for the left-hand triangle.  This proves part~(ii).
\end{proof}

A difficulty in generalizing the previous theorem to $\Spdot[n]$
arises in part~(i):
One would like to continue the construction of the $T^{*}\subdot$ down
the rows of $\Spdot[n]$, but the trick in the previous argument fails
at this stage because the map from the pushout to the appropriate
target object $B_{1,2}$ need not be a weak cofibration.  (In the case
when all maps are weak cofibrations, Lemma~\ref{lemfactuse}
immediately implies that $\Spdot[n]$ has a USE of 
factorizations.)  On the other hand, the analogue of part~(ii) for
$\Spdot[n]$ follows from part~(ii) for $F_{n}$.

\begin{cor}[{\cite[2.9]{BlumbergMandell}}]\label{corspdot}
Let $\aC$ be a Waldhausen category with a USE of
factorizations for weak cofibrations. 
The forgetful functor $\w\Sdot[n]\aC\to \w\Spdot[n]\aC$ induces a weak
equivalence on nerves.
\end{cor}

\begin{proof}
Let $M=U^{n}\times_{N\w F'_{n}\aC}N\w\Spdot[n]\aC$, where $U^{n}\to
N\w F'_{n}\aC$ is as in the proof of Theorem~\ref{thmfpdot}.  We have a
universal simplicial equivalence $M\to N\w\Spdot[n]\aC$, and pushout
induces a map $\Phi \colon M\to N\w\Sdot[n]\aC$ as in the proof of Theorem~2.9 in
\cite{BlumbergMandell}.  Looking at the diagram,
\[
\xymatrix{%
M\times_{N\w\Spdot[n]\aC}N\w\Sdot[n]\aC\ar@{->>}[d]_{\htp}\ar[r]&
M\ar@{->>}[d]_{\htp}\ar@{.>}[dl]\\
N\w\Sdot[n]\aC\ar[r]&N\w\Spdot[n]\aC
}
\]
The argument of \cite[2.9]{BlumbergMandell} generalizes to show that
the left-hand triangle commutes up to simplicial homotopy and the
right-hand triangle commutes up to generalized simplicial homotopy.
\end{proof}

\subsection{Iterating $F'\subdot$ and $\Spdot$}

Corollary~\ref{corspdot} provides a space-level comparison of the
(delooped) $K$-theory spaces provided by the $\Sdot$ and $\Spdot$
constructions.  We next extend this to a spectrum-level comparison by
comparing the iterated $\Sdot$ construction with an $\Spdot$ analogue.
Because we can not prove the analogue of part~(i) of
Theorem~\ref{thmfpdot} for $\Spdot$ (the generalization of
\cite[2.8]{BlumbergMandell}), we need to take a direct approach to the
construction of iterated $\Spdot$.  Again, we find it convenient to
start by examining $F'\subdot$.

\begin{defn}\label{deffpdotq}
Let $\aC$ be a Waldhausen category with a USE of factorizations for weak
cofibrations.  For $q,n_{1},\dotsc,n_{q}\geq 0$, let $F^{(q)}_{n_{1},\dotsc,n_{q}}\aC$ be the
Waldhausen category $F_{n_{1}}F_{n_{2}}\dotsb F_{n_{q}}\aC$ and let
$F^{\prime(q)}_{n_{1},\dotsc,n_{q}}\aC$ be the 
Waldhausen category $F'_{n_{1}}F'_{n_{2}}\dotsb F'_{n_{q}}\aC$.
Similarly, let $S^{(q)}_{n_{1},\dotsc,n_{q}} \aC$ denote the
Waldhausen category $S_{n_{1}} S_{n_{2}} \dotsb S_{n_{q}}\aC$ given by
iterating the $\Sdot$ construction.
\end{defn}

Combining the construction of the universal simplicial
equivalence $U^{n}\to \w F'_{n}\aC$ in part~(ii) of
Theorem~\ref{thmfpdot} with the construction in part~(i) of
Theorem~\ref{thmfpdot} proves the following theorem.

\begin{thm}\label{thmuq}
Let $\aC$ be a Waldhausen category with a USE of factorizations of
weak cofibrations.  There exist weak equivalences 
\[
u\colon U^{(q)}_{n_{1},\dotsc,n_{q}}\to
N\w F^{\prime(q)}_{n_{1},\dotsc,n_{q}}\aC\qquad \text{and}\qquad 
p\colon U^{(q)}_{n_{1},\dotsc,n_{q}}\to
N\w F^{(q)}_{n_{1},\dotsc,n_{q}}\aC
\]
and a simplicial homotopy from the composite map 
\[
U^{(q)}_{n_{1},\dotsc,n_{q}}\to
N\w F^{(q)}_{n_{1},\dotsc,n_{q}}\aC\to 
N\w F^{\prime(q)}_{n_{1},\dotsc,n_{q}}\aC
\]
to $u$; moreover $u$ is a universal simplicial quasifibration.
\end{thm}

\begin{cor}\label{corfpqmain}
Let $\aC$ be a Waldhausen category with a USE of factorizations of
weak cofibrations. 
An $n_{1}\times \dotsb \times n_{q}$ rectangle of maps in $\aC$ is an
object of $F^{\prime(q)}_{n_{1},\dotsc,n_{q}} \aC$ if and only there
exists a weak equivalence (of rectangular diagrams) to it from an
object of $F^{(q)}_{n_{1},\dotsc,n_{q}} \aC$. 
\end{cor}

This corollary makes the categories
$F^{\prime(q)}_{n_{1},\dotsc,n_{q}} \aC$
significantly more tractable.  For example, it follows that the usual
symmetric group action on $F^{(q)}_{n,\dotsc,n} \aC$ extends to
$F^{\prime(q)}_{n,\dotsc,n} \aC$.  

We now construct 
categories $\Spdotq \aC$ that play the role of the iterated $\Spdot$
construction.  For this, recall that $\Ar[n]$
denotes the category with objects $(i,j)$ for $0\leq i\leq j\leq n$ and
a unique map $(i,j)\to (i',j')$ for $i\leq i'$ and $j\leq j'$.  We
write an object of $\Ar[n_{1}]\times \dotsb \times \Ar[n_{q}]$ as
$(i_{1},j_{1};\dotsc;i_{q},j_{q})$.

\begin{defn}\label{defspdotq}
Let $\aC$ be a Waldhausen category with a USE of factorizations of
weak cofibrations. Let
$\Spdotq[n_{1},\dotsc,n_{q}]\aC$ be the full subcategory of functors 
\[
\Ar[n_{1}] \times \dotsb \times \Ar[n_{q}]\to \aC
\]
such that: 
\begin{itemize}
\item The initial map $*\to A_{i_{1},j_{1};\dotsc;i_{q},j_{q}}$ is a
weak equivalence whenever $i_{k}=j_{k}$ for some $k$.
\item The $n_{1}\times \dotsb \times n_{q}$ rectangular subdiagram
$A_{0,j_{1};\dotsb;0,j_{q}}$ is an object of
$F^{\prime(q)}_{n_{1},\dotsc,n_{q}} \aC$. 
\item For every object 
$(i_{1},j_{1};\dotsc;i_{q},j_{q})$ in $\Ar[n_{1}]\times \dotsb \times
\Ar[n_{q}]$, the square
\[  \xymatrix@-1pc{%
A_{0,i_{1};\dotsc;0,i_{q}}\ar[r]\ar[d]&A_{0,j_{1};\dotsc;0,j_{q}}\ar[d]\\
A_{i_{1},i_{1};\dotsc;i_{q},i_{q}}\ar[r]&A_{i_{1},j_{1};\dotsc;i_{q},j_{q}}
} \]
is homotopy cocartesian.
\end{itemize}
The subcategory $\w\Spdotq[n_{1},\dotsc,n_{q}]$ consists of the maps in
$\Spdot[n_{1},\dotsc,n_{q}]$ that are objectwise weak equivalences.
\end{defn}

We understand $\Spdotmac{0}{}\aC$ to be $\aC$ and we see that
$\Spdotmac{1}{n}\aC$ is $\Spdot[n] \aC$.  As an easy consequence of
Corollary~\ref{corfpqmain}, $\Spdotq \aC$ and $\w\Spdotq \aC$ assemble
into simplicial categories.  Likewise, it follows from
Corollary~\ref{corfpqmain} that $\Sdotq \aC$ and $\w\Sdotq \aC$ are
simplicial subcategories of $\Spdotq \aC$ and $\w\Spdotq \aC$.  The
argument for Corollary~\ref{corspdot} generalizes to prove the
following theorem, which provides the replacement for Theorems~2.8
and~2.9 of \cite{BlumbergMandell}.

\begin{thm}
The forgetful functor $\w\Sdotq \aC\to \w\Spdotq \aC$ induces a weak equivalence
on nerves.
\end{thm}

\begin{cor}\label{corusewkexact}
Let $F\colon \aC\to \aD$ be a weakly exact functor.  If $\aD$ has a
USE of factorizations of weak cofibrations, then $F$ induces a map of
$K$-theory spectra.
\end{cor}

\subsection{Proof of Lemma~\ref{lemfactuse}}
For a map $f\colon A\to B$ in $\aC$, say that a mapping cylinder 
\[
\xymatrix@R-1pc{%
&B\ar[d]\ar[dr]^{=}\\
A\ar@{ >->}[r]\ar@{..>}@/_1em/[rr]_{f}&X\ar[r]^{\htp}&B,
}
\]
is a \term{strong} mapping cylinder when the map $A\vee B\to X$ is a
cofibration.  Let $\MCp\aC$ be the full subcategory of the category of mapping
cylinders $\MC\aC$ consisting of the strong mapping cylinders.  For a
fixed map $f$ in $\aC$, we write $\MCp f$ for the category of strong
mapping cylinders for $f$; this is the subcategory of $\MCp\aC$
consisting of the objects that go to the object $f$ of $\Ar\aC$ and
the maps that go to the identity map of $f$ in $\Ar\aC$ under the
forgetful functor 
$\MCp\aC\to\Ar \aC$.

Now assume that $\aC$ admits factorization.
Consider the bisimplicial set $T_{\ssdot,\ssdot}$ whose 
set of $p,q$-simplices $T_{p,q}$ consists of the commuting diagrams in
$\MCp\aC$ 
\[
\xymatrix@-1pc{%
\mathbf{X}_{0,0}\ar[r]\ar[d]&
\mathbf{X}_{1,0}\ar[r]\ar[d]&
\dotsb\ar[r]
&\mathbf{X}_{p-1,0}\ar[r]\ar[d]
&\mathbf{X}_{p,0}\ar[d]\\
\mathbf{X}_{0,1}\ar[r]\ar[d]&
\mathbf{X}_{1,1}\ar[r]\ar[d]&
\dotsb\ar[r]
&\mathbf{X}_{p-1,1}\ar[r]\ar[d]
&\mathbf{X}_{p,1}\ar[d]\\
\vdots\ar[d]&\vdots\ar[d]&&\vdots\ar[d]&\vdots\ar[d]\\
\mathbf{X}_{0,q}\ar[r]&
\mathbf{X}_{1,q}\ar[r]&
\dotsb\ar[r] 
&\mathbf{X}_{p-1,q}\ar[r]
&\mathbf{X}_{p,q}
}
\]
where each of the vertical maps $\mathbf{X}_{i,j}\to
\mathbf{X}_{i,j+1}$ forgets to an identity 
morphism $\id_{f}$ in the category $\Ar\aC$.  We have a map from the diagonal
simplicial set to $N\subdot\MC$ that takes the diagram above (for
$p=q$) to the sequence
\[
\mathbf{X}_{0,0}\to \mathbf{X}_{1,1}\to \dotsb \to \mathbf{X}_{q,q}.
\]
We show that the composite map to $N\subdot\Ar\aC$ is a universal
simplicial equivalence.

The composite map $T\subdot\to N\subdot\Ar\aC$ is the diagonal of a
map of the bisimplicial sets $T_{p,q}\to N_{p}\Ar\aC$ induced by the
forgetful functor $\MCp\to\Ar\aC$, where
we regard $N_{p}\Ar\aC$ as constant in the second simplicial
direction.  Since we can regard any simplicial map $Z\subdot \to
N\subdot\Ar\aC$ as a bisimplicial map, constant in the second
simplicial direction,
and since the diagonal preserves pullbacks, it suffices to show that
for each $p$, the map of simplicial sets from $T_{p,\ssdot}$ to the
constant simplicial set $N_{p}\Ar\aC$ is a weak equivalence.
Moreover, identifying the category $N_{p}\Ar\aC$ with the category
$\Ar N_{p}\aC$, this amounts to showing that for a map $f$ between
objects of $N_{p}\aC$, the category of strong mapping cylinders for
$f$ has contractible nerve.  Since the Waldhausen category $N_{p}\aC$
admits factorizations when $\aC$ does, it suffices to treat the case
$p=0$. 

Thus, we need to show that for every $f\colon A\to B$ in $\aC$, the
category $\MCp f$ has contractible nerve.  We view $\MCp f$ as a
subcategory of the category $\aC\mid f$ of diagrams 
\[
A\vee B\to X \to B
\]
in $\aC$ such that the composite map $A\vee B\to B$ is $f+\id_{B}$.
We apply Waldhausen's argument for the Approximation Theorem as
generalized in \cite[A.2]{Schlichting}.  First, observe that
factorization implies that $\MCp f$ is nonempty.
Since after suitable simplicial approximation and subdivision any
homotopy class of maps from a sphere to the geometric realization of
$\MCp$ is represented by the geometric 
realization of a functor from a finite partially ordered set into
$\MCp$, it suffices to show that any functor $\alpha$ from a
finite partially ordered set $\aP$ to $\MCp$ admits a zigzag of
natural transformations to a constant functor 
\cite[A.10]{Schlichting}.  The key idea is
to inductively apply factorization so that colimits over sub-posets
exist as iterated pushouts over cofibrations; this approach
constructs a functor $\beta \colon \aP\to \MCp f$ and natural
transformation $\beta \to \alpha$ 
such that the colimit of $\beta \colon \aP\to \aC\mid f$ (exists and)
can be constructed as an iterated pushout over cofibrations
\cite[A.6]{Schlichting}.  This colimit is not an element in $\MCp f$,
but applying factorization in $\aC$, gives an object $\mathbf{X}$ in $\MCp f$
and a map $\colim_{\aP}\beta \to \mathbf{X}$ in $\aC\mid f$.  This then gives a
natural transformation from $\beta$ to the constant functor on $\mathbf{X}$.
\qed

\section{Generalizing to the non-functorial case}
\label{appmain}

In this appendix, we go section by section through the paper and
indicate the changes in statements and proofs needed for the case
when the required factorizations are not functorial.

\subsection{Introduction}

Statements are made in the non-functorial case.

\subsection{Weakly exact functors}

Corollary~\ref{corusewkexact} substitutes for Theorem~\ref{thmweakexact}
for categories that have a USE of factorizations for weak cofibrations
in place of FFWC.  Lemma~\ref{lemfactuse} implies that Waldhausen
categories that admit factorization in particular have a USE of factorization
for weak cofibrations.

The hypothesis of FMCWC in Theorem~\ref{thmhocoend} generalizes to the
hypothesis of a USE of 
mapping cylinders for weak cofibrations.  Here is a full statement:

\begin{thm}\label{appthmhocoend}
Let $\aC$ be a saturated Waldhausen category that has a USE of mapping
cylinders for weak cofibrations.
\begin{enumerate}
\item
For $n>1$, the nerve of $\w\Spdot[n] \aC$ is weakly equivalent to
the homotopy coend 
\[
\hocoendlim_{(X_{1},\dotsc,X_{n})\in \w\aC^{n}}\Lco\aC(X_{n-1},X_{n})\times \dotsb \times
\Lco\aC(X_{1},X_{2}),
\]
naturally in weakly exact functors.  
\item The nerve of $\w\aC$ is weakly equivalent to the disjoint
union of $B\Laut{X}$ over the weak equivalence classes of objects of
$\aC$.
\item
For $n\geq 1$, the nerve of $\w\Spdot[n] \aC$ is weakly equivalent
to the total space
of a fibration where the base is 
the disjoint union of 
\[
B\Laut{X_{n}}\times \dotsb \times B\Laut{X_{1}}
\]
over $n$-tuples of weak equivalences
classes of objects of $\aC$, and the fiber is equivalent to
\[
\Lco\aC(X_{n-1},X_{n})\times \dotsb \times
\Lco\aC(X_{1},X_{2}).
\]
for $n>1$ and contractible for $n=1$.
\end{enumerate}
\end{thm}

\subsection{Outline of the proof of Theorem~\ref{main}}

The outline proceeds somewhat differently without functorial
factorizations.  First, applying Lemma~\ref{lemfactuse}, let
$T\subdot\to N\Fac\aC$ be a USE of factorizations
for $\aC$.  Next, in place of cone and suspension functors, we define a
category of cone and suspension objects.

\begin{defn}
Let $E\aC$ be the Waldhausen subcategory of $\Sdot[2]\aC$ of objects 
\[
\xymatrix@-1pc{%
X\ar@{ >->}[r]&C\ar[r]&\Sigma
}
\]
such that the initial map $*\to C$ is a weak equivalence.  Let $E'\aC$
be the Waldhausen subcategory of $\Spdot[2]\aC$ of objects $\{A_{i,j}\}$
such that the initial map $*\to A_{0,2}$ is a weak equivalence.
\end{defn}

We have three exact functors $E\aC\to \aC$, the \term{forgetful} functor,
the \term{cone} functor, and the \term{suspension} functor defined by
sending the object pictured above to $X$, $C$, and $\Sigma$
respectively.  We refer to the corresponding functors $E'\aC\to \aC$
by the same names; specifically, for an object $\{A_{i,j}\}$ in
$E'\aC$, the forgetful functor sends it
to $A_{0,1}$, the cone functor sends it to $A_{0,2}$,
and the suspension functor sends it to $A_{1,2}$. 

The Waldhausen categories $E\aC$ and $E'\aC$ inherit factorizations
from $\aC$.  The forgetful functors $E\aC\to\aC$ and $E'\aC\to\aC$
satisfy Waldhausen's approximation property.  Using
Schlichting's extension of Waldhausen's Approximation Theorem
\cite[A.2]{Schlichting}, 
Theorem~\ref{thmhh}, and Theorem~\ref{propapprox}, we obtain the 
following theorem.

\begin{thm}\label{appthmforg}
The forgetful functors $E\aC\to \aC$ and $E'\aC\to \aC$ are
DK-equivalences and induce weak
equivalences on $K$-theory.
\end{thm}

Applying Waldhausen's Additivity Theorem, we obtain the following result.

\begin{cor}\label{appcorsusp}
The suspension functors $E\aC\to \aC$ and $E'\aC\to \aC$ induce weak
equivalences on $K$-theory.
\end{cor}

For the purposes of generalizing the arguments in
Section~\ref{secpfmain}, we say that an object of $\aC$ is a
suspension object if it is in the image of the suspension functor $E'\aC\to
\aC$.  Then a weakly exact functor between Waldhausen categories that
admit factorization sends suspension objects to suspension objects.
Theorem~\ref{thmprehococart} (proved in Section~\ref{sechococart})
then implies the following result in this language.

\begin{cor}\label{appcordk}
With hypotheses as in Theorem~\ref{main}, the map $LF\colon
L\aC(X,Y)\to L\aD(FX,FY)$ is a weak equivalence when $X$ is a
suspension object. 
\end{cor}

Define $K'\aC$ as the homotopy colimit of the diagram
\[
\xymatrix@-1pc{%
&N(\w\Spdot E'\aC)\ar[dl]_{\htp}\ar[dr]^{\htp}&
&N(\w\Spdot E'\aC)\ar[dl]_{\htp}\ar[dr]^{\htp}&&\dotsb\ar[dl]^-{\htp} \\
N(\w\Spdot\aC)&&
N(\w\Spdot\aC)&&
N(\w\Spdot\aC)&\dotsb 
}
\]
where the leftward arrows are induced by the forgetful functor and the
rightward arrows are induced by the suspension functor.  We have the
corresponding construction $K'\aD$ for $\aD$ and $F$ induces a map
$K'\aC\to K'\aD$.  By Theorem~\ref{appthmforg} and
Corollary~\ref{appcorsusp}, all the maps in the diagram are weak
equivalences and it follows that $K'\aC\to K'\aD$ models the induced
map on delooped $K$-theory spaces.  The proof of Theorem~\ref{main} is
completed by showing that this map is a weak equivalence.

The argument in Section~\ref{secpfmain} generalizes as follows.
According to Theorem~\ref{thmhocoend}, the
commuting square of functors
\[
\xymatrix{%
\w\Spdot[n]\aC\ar[d]_{F}&\w\Spdot[n]E'\aC\ar[l]_{\htp}\ar[r]^{\htp}\ar[d]_{F}&
\w\Spdot[n]\aC\ar[d]_{F}\\
\w\Spdot[n]\aD&\w\Spdot[n]E'\aD\ar[l]^{\htp}\ar[r]_{\htp}&
\w\Spdot[n]\aD
}
\]
(where the lefthand functors are the forgetful functor and the
righthand maps are the suspension functors) induces on nerves a map
modeled by the diagram
\begin{small}
\[
\xymatrix@C-1pc{%
\hocoend L\aC(X_{1},\dotsc,X_{n})\ar[d]_{LF}
&\hocoend LE'\aC(\mathbf{X}_{1},\dotsc,\mathbf{X}_{n})\ar[d]_{LF}
\ar[l]\ar[r]
&\hocoend L\aC(X_{1},\dotsc,X_{n})\ar[d]_{LF}\\
\hocoend L\aD(Y_{1},\dotsc,Y_{n})
&\hocoend LE'\aD(\mathbf{Y}_{1},\dotsc,\mathbf{Y}_{n})
\ar[l]\ar[r]
&\hocoend L\aD(Y_{1},\dotsc,Y_{n})
}
\]%
\end{small}%
where the homotopy coends are over $(X_{1},\dotsc,X_{n})\in \w\aC^{n}$,
$(\mathbf{X}_{1},\dotsc,\mathbf{X}_{n})\in (wE'\aC)^{n}$,
$(Y_{1},\dotsc,Y_{n})\in \w\aD^{n}$, and
$(\mathbf{Y}_{1},\dotsc,\mathbf{Y}_{n})\in (wE'\aD)^{n}$.  The right
hand square factors as
\begin{small}
\[
\xymatrix@C-1pc{%
\hocoend LE'\aC(\mathbf{X}_{1},\dotsc,\mathbf{X}_{n})\ar[d]_{LF}
\ar[r]
&\hocoend L\aC(C_{1},\dotsc,C_{n})\ar[d]_{LF}\ar[r]
&\hocoend L\aC(X_{1},\dotsc,X_{n})\ar[d]_{LF}\\
\hocoend LE'\aD(\mathbf{Y}_{1},\dotsc,\mathbf{Y}_{n})
\ar[r]
&\hocoend L\aD(D_{1},\dotsc,D_{n})\ar[r]
&\hocoend L\aD(Y_{1},\dotsc,Y_{n})
}
\]%
\end{small}%
where the middle homotopy coends are over $n$-tuples of suspension objects in
$\w\aC$ and $\w\aD$.  Corollary~\ref{appcordk} then
implies that the middle vertical arrow is a weak equivalence.  Going
back to the homotopy colimit defining $K'\aC$ and $K'\aD$, we see that
the map $K'\aC\to K'\aD$ is a weak equivalence.  This completes the
proof of Theorem~\ref{main}.

\subsection{Universal simplicial quasifibrations}

No statements or arguments in this section involve factorization.

\subsection{Homotopy calculi of fractions and mapping cylinders}

The hypothesis of FMCWC in Theorems~\ref{thmHCLF} and~\ref{thmwordcomp} (and
implicitly Lemma~\ref{lemHCLF}, which shares the hypothesis of
Theorem~\ref{thmHCLF}) generalizes to the hypothesis of USE of mapping
cylinders for weak cofibrations.  The statements become:

\begin{thm}\label{appthmHCLF}
Let $\aC$ be a Waldhausen category with a USE of mapping cylinders for
weak cofibrations.  Then $\aC$,
$\wco\aC$, and $\w\aC$ admit homotopy
calculi of left fractions.
\end{thm}

\begin{thm}\label{appthmwordcomp}
Let $\aC$ be a Waldhausen category with a USE of mapping cylinders for
weak cofibrations.  Then the maps
\begin{gather*}
\WC(A,B)_{\wco}\to \Lco\aC(A,B), \qquad 
\WCW(A,B)_{\wco}\to \Lco\aC(A,B),\\
\WC(A,B)_{\w}\to \Lw\aC(A,B),\qquad\text{and}\qquad 
\WCW(A,B)_{\w}\to \Lw\aC(A,B)
\end{gather*}
are weak equivalences
\end{thm}

The proof of Theorem~\ref{thmHCLF} from Lemma~\ref{lemHCLF} and
Theorem~\ref{thmwordcomp} from Theorem~\ref{thmHCLF} generalize
immediately to Theorems~\ref{appthmHCLF} and~\ref{appthmwordcomp}.
The proof of Lemma~\ref{lemHCLF} is modified as follows:

\begin{proof}[Proof of Lemma~\ref{lemHCLF}]
Let $T\subdot\to N\Fac\aC$ be a
USE of mapping cylinders for weak cofibrations.  Let $T\subdot^{i,j}(A,B)$
be the pullback
\[
\xymatrix{%
T^{i,j}\subdot\ar[r]\ar@{->>}[d]_{\htp}
&T\subdot\ar@{->>}[d]^{\htp}\\
N\Wi \C^i \Wi \C^j(A,B)\ar[r]_(.65){\phi}&N\wcAr\aC,
}
\]
where the map $\phi$ is induced by the functor that takes the object 
\[
\xymatrix@-1pc{%
A\ar[r]&Y_{j}\ar[r]&\dotsb\ar[r]&Y_{1}&Z\ar[l]_{f}^{\htp}\ar[r]
&X_{i}\ar[r]&\dotsb\ar[r]&X_{1}&B\ar[l]_(.4){\htp}
}
\]
of $\Wi \C^i \Wi \C^j(A,B)$ to the object $f$ of $\wcAr\aC$.
Then the pushout construction in the proof of this lemma in
Section~\ref{sechclf} constructs a map 
\[
T^{i,j}\subdot(A,B)\to N\Wi\C^{i}\C^{j}(A,B)
\]
and the natural transformations there give
simplicial homotopies to make the diagram
\[
\xymatrix@C-1pc{%
T^{i,j}\subdot(A,B)\times_{N\Wi \C^i \Wi \C^j(A,B)}
N\Wi\C^{i}\C^{j}(A,B)\ar@{->>}[d]_{\htp}\ar[r]&
T^{i,j}\subdot(A,B)\ar@{->>}[d]_{\htp}\ar@{.>}[dl]\\
N\Wi\C^{i}\C^{j}(A,B)\ar[r]&N\Wi \C^i \Wi \C^j(A,B)
}
\]
commute up to generalized simplicial homotopy.
\end{proof}

\subsection{Homotopy cocartesian squares in {W}aldhausen categories}

We note that the hypothesis of a USE of mapping cylinders for weak
cofibrations also implies
the existence of mapping cylinders for weak cofibrations and HCLF.
The statements and proofs in this section are unchanged.

\subsection{Proof of Theorems~\ref{thmhocoendone}, \ref{thmhocoendtwo},
and~\ref{thmhocoend}}

Proposition~\ref{propfprime} generalizes to the case of a USE of
factorizations for weak cofibrations by the work in the previous
appendix.  Combining Lemma~\ref{lemfactuse} as well, we have the
following statement:

\begin{prop}\label{apppropfprime}
If $\aC$ has a USE of factorizations of weak cofibrations or $\aC$
admits factorizations, then the forgetful functor $\w\Spdot[n]\aC\to
\w F'_{n-1}\aC$ induces a weak equivalence on nerves.
\end{prop}

\subsection{Proof of Theorem~\ref{thmhh}}

Most of the statements and proofs in this section are written in terms of
non-functorial factorization.  The only exception is
Proposition~\ref{propunderiscof}, which we need to prove in the
non-functorial case.

\begin{proof}[Proof of Proposition~\ref{propunderiscof}]
Let $A\to B$ and $A\to C$ be objects in $\COF{A}$.  For each word
$\Word$ in $\C$ and $\Wi$, we have categories $\Word_{\co}(B,C)$ and
$\Word(B,C)$ of diagrams in $\COF{A}$ and $\aC\bs A$, respectively, as
in Section~\ref{sechclf}; it suffices to show that the inclusion
$\Word_{\co}(B,C)\to \Word(B,C)$ induces a weak equivalence on
nerves. 

Fix a word $\Word$; the argument for Lemma~\ref{lemtriv}
shows that it suffices to consider the case where $\Word$ contains no
subword of the form $\Wi\Wi$.  Any two letter subword of $\Word$ is
then of the form $\C\Wi$, $\C\C$, or $\Wi\C$:
\[
\xymatrix@-1pc{%
\C\Wi:&
B\ar@{{}-{}}[r]&\cdots\ar@{{}-{}}[r]
&X_{i+1}&X_{i}\ar[r]\ar[l]_{\htp}&X_{i-1}\ar@{{}-{}}[r]
&\cdots\ar@{{}-{}}[r]&C\\
\C\C:&
B\ar@{{}-{}}[r]&\cdots\ar@{{}-{}}[r]
&X_{i+1}\ar[r]&X_{i}\ar[r]&X_{i-1}\ar@{{}-{}}[r]
&\cdots\ar@{{}-{}}[r]&C\\
\Wi\C:&
B\ar@{{}-{}}[r]&\cdots\ar@{{}-{}}[r]
&X_{i+1}\ar[r]&X_{i}&X_{i-1}\ar[l]_{\htp}\ar@{{}-{}}[r]
&\cdots\ar@{{}-{}}[r]&C
}
\]
In each of these cases we call $X_{i}$ the \term{pivot} of the subword.
Let $\Word_{\co(\C\Wi)}(B,C)$ be the
full subcategory of $\Word(B,C)$ consisting of those objects where
the structure maps $A\to X_{i}$ is a cofibration whenever $X_{i}$ is
the pivot of a $\C\Wi$ subword.  Likewise, let
$\Word_{\co(\C\Wi,\C\C)}(B,C)$ be the full subcategory where the
structure map is a cofibration for the pivots of all $\C\Wi$ and $\C\C$
subwords.  We then have inclusions of full subcategories
\[
\Word_{\co}(B,C)\to
\Word_{\co(\C\Wi,\C\C)}(B,C)\to
\Word_{\co(\C\Wi)}(B,C)\to 
\Word(B,C)
\]
and it suffices to show that each of these induces weak equivalences
on nerves.

We start with the inclusion $\Word_{\co(\C\Wi)}(B,C)\to
\Word(B,C)$.  By Lemma~\ref{lemfactuse}, we have a USE of
factorizations $T\subdot\to N\Fac\aC$. 
For each subword $\C\Wi$, we have a functor
$\Word(B,C)\to \Ar\aC$ sending the object in $\Word(B,C)$ pictured
above to the object $A\to X_{i}$ in $\Ar\aC$.  Let $U\subdot$ be the
pullback of the diagram below.
\[
\xymatrix{%
U\subdot\ar[r]\ar@{->>}[d]_{\htp}
&T\subdot\times \dotsb \times T\subdot\ar@{->>}[d]_{\htp}\\
N\Word(B,C)\ar[r]&N\Ar\aC\times \dotsb \times N\Ar\aC
}
\]
Then from the map $T\subdot\to N\Fac\aC$, we get a map $U\subdot\to
\Word_{\co(\C\Wi)}(B,C)$ and simplicial homotopies making the diagram
\[
\xymatrix{%
U\subdot\times_{N\Word(B,C)}N\Word_{\co(\C\Wi)}(B,C)
\ar@{->>}[d]_{\htp}\ar[r]&
U\subdot
\ar@{->>}[d]_{\htp}\ar@{.>}[dl]\\
N\Word_{\co(\C\Wi)}(B,C)\ar[r]&N\Word(B,C)
}
\]
commute up to simplicial homotopy.

For the inclusion $\Word_{\co(\C\Wi,\C\C)}(B,C)\to
\Word_{\co(\C\Wi)}(B,C)$, we work by induction.  Let $\Word_{i}$ be
the full subcategory of $\Word_{\co(\C\Wi)}(B,C)$ consisting of the
objects for which the structure maps are cofibrations for the pivots
of the last $i$ subwords of the form $\C\C$.  Then
$\Word_{\co(\C\Wi)}(B,C)=\Word_{0}(B,C)$ and
$\Word_{\co(\C\Wi,\C\C)}(B,C)$ is $\Word_{n}(B,C)$ for some $n$; we
show that the inclusions 
\[
\Word_{i+1}(B,C)\to \Word_{i}(B,C)
\]
induce weak equivalences on nerves.  In $\Word_{i}(B,C)$, consider the
pivot of the $(i+1)$-st from last subword of the form $\C\C$:
\[
\xymatrix@-1pc{%
B\ar@{{}-{}}[r]&\cdots\ar@{{}-{}}[r]
&X_{j+1}\ar[r]&X_{j}\ar[r]&X_{j-1}\ar@{{}-{}}[r]
&\cdots\ar@{{}-{}}[r]&C\\
}
\]
Either $X_{j+1}=B$, or $X_{j+1}$ is a pivot of a subword $\C\Wi$, or
$X_{j+1}$ is the pivot of the $i$-th from last subword of the form
$\C\C$.  In any of these cases, the structure map $A\to X_{j+1}$ is a
cofibration.  Using the functor $\Word_{i}(B,C)\to \Ar\aC$ sending the
object in $\Word_{i}(B,C)$ pictured above to the object $X_{j+1}\to
X_{j}$ of $\Ar\aC$, we get the solid arrow diagram below.
\[
\xymatrix{%
T\subdot\times_{N\Ar\aC}N\Word_{i+1}(B,C)\ar@{->>}_{\htp}[d]\ar[r]&
T\subdot\times_{N\Ar\aC}N\Word_{i}(B,C)\ar@{->>}_{\htp}[d]\ar@{.>}[dl]\\
N\Word_{i+1}(B,C)\ar[r]&N\Word_{i}(B,C)
}
\]
The map $T\subdot\to N\Fac\aC$ induces the dotted arrow and simplicial
homotopies making both triangles commute up to simplicial homotopy.

For the inclusion $\Word_{\co}(B,C)\to \Word_{\co(\C\Wi,\C\C)}(B,C)$,
consider the subwords of the form $\Wi\C$
\[
\xymatrix@-1pc{%
B\ar@{{}-{}}[r]&\cdots\ar@{{}-{}}[r]
&X_{k+1}\ar[r]&X_{k}&X_{k-1}\ar[l]_{\htp}\ar@{{}-{}}[r]
&\cdots\ar@{{}-{}}[r]&C.
}
\]
If the subword is the final two symbols of $\Word$, then $X_{k+1}$ is
$B$; if not, then the next letter in $\Word$ is $\Wi$ or $\C$, and $X_{k+1}$
is in the middle of a $\C\Wi$ or $\C\C$ subword.  In either case, the structure
map $A\to X_{k+1}$ is a cofibration.  Likewise the structure map $A\to
X_{k-1}$ is a cofibration.  We therefore obtain a functor
$\Word_{\co(\C\Wi,\C\C)}(B,C)\to \Ar\aC$ taking the object of 
$\Word_{\co(\C\Wi,\C\C)}(B,C)$ pictured above to the object
$X_{k-1}\cup_{A}X_{k+1}\to X_{k}$ of $\Ar\aC$.  The same argument as
in the $\C\Wi$ subword argument then shows that the inclusion of 
$\Word_{\co}(B,C)$ in $\Word_{\co(\C\Wi,\C\C)}(B,C)$ induces a weak
equivalence on nerves, and completes the proof.
\end{proof}

\subsection{Proof of Theorem~\ref{propapprox}}

No statements or arguments in this section involve functoriality of
the factorizations.


\bibliographystyle{plain}

\begin{thebibliography}{10}

\bibitem{BlumbergMandell}
Andrew~J. Blumberg and Michael~A. Mandell, \emph{The localization sequence for
  the algebraic {$K$}-theory of topological {$K$}-theory}, Acta Math.
  \textbf{200} (2008), no.~2, 155--179. \MR{MR2413133}

\bibitem{CisinskiUnPub}
Denis-Charles Cisinski, \emph{Invariance de la k-th\'eorie par \'equivalences
  d\'eriv\'ees}, J. {$K$}-theory,to appear.

\bibitem{DuggerShipley}
Daniel Dugger and Brooke Shipley, \emph{{$K$}-theory and derived equivalences},
  Duke Math. J. \textbf{124} (2004), no.~3, 587--617. \MR{MR2085176
  (2005e:19005)}

\bibitem{DKHammock}
W.~G. Dwyer and D.~M. Kan, \emph{Calculating simplicial localizations}, J. Pure
  Appl. Algebra \textbf{18} (1980), no.~1, 17--35. \MR{MR578563 (81h:55019)}

\bibitem{DKModel}
\bysame, \emph{Function complexes in homotopical algebra}, Topology \textbf{19}
  (1980), no.~4, 427--440. \MR{MR584566 (81m:55018)}

\bibitem{DKHS}
W.G. Dwyer, P.S. Hirschhorn, D.M. Kan, and J.H. Smith, \emph{Homotopy limit
  functors on model categories and homotopical categories}, Mathematical
  Surveys and Monographs, vol. 113, American Mathematical Society, Providence,
  RI, 2004. \MR{MR2102294 (2005k:18027)}

\bibitem{Goodwillie-Davis}
T.~Goodwillie, URL: \texttt{http://www.lehigh.edu/\char126 dmd1/tg516.txt}.

\bibitem{Grothpursue}
Alexander Grothendieck, \emph{Pursuing stacks}, manuscript, 1983.

\bibitem{Grothderiv}
\bysame, \emph{D\'{e}rivateurs}, manuscript, 1983--1990.

\bibitem{Hellerhom}
Alex Heller, \emph{Homotopy theories}, Mem. Amer. Math. Soc. \textbf{71}
  (1988), no.~383, vi+78. \MR{MR920963 (89b:55013)}

\bibitem{Kellerderiv}
Bernhard Keller, \emph{Derived categories and universal problems}, Comm.
  Algebra \textbf{19} (1991), no.~3, 699--747. \MR{MR1102982 (92b:18010)}

\bibitem{MaltK}
Georges Maltsiniotis, \emph{La {$K$}-th\'eorie d'un d\'erivateur triangul\'e},
  Categories in algebra, geometry and mathematical physics, Contemp. Math.,
  vol. 431, Amer. Math. Soc., Providence, RI, 2007, pp.~341--368. \MR{2342836
  (2008i:18008)}

\bibitem{MandellHH}
Michael~A. Mandell, \emph{Equivalence of simplicial localizations of closed
  model categories}, J. Pure Appl. Algebra \textbf{142} (1999), no.~2,
  131--152. \MR{2000i:55050}

\bibitem{Neeman1}
Amnon Neeman, \emph{{$K$}-theory for triangulated categories. {I}({A}).
  {H}omological functors}, Asian J. Math. \textbf{1} (1997), no.~2, 330--417.
  \MR{MR1491990 (99m:18008a)}

\bibitem{QuillenAK}
Daniel Quillen, \emph{Higher algebraic {$K$}-theory. {I}}, Algebraic
  $K$-theory, I: Higher $K$-theories (Proc. Conf., Battelle Memorial Inst.,
  Seattle, Wash., 1972), Springer, Berlin, 1973, pp.~85--147. Lecture Notes in
  Math., Vol. 341. \MR{MR0338129 (49 \#2895)}

\bibitem{SchlichtingDerivedIneq}
Marco Schlichting, \emph{A note on {$K$}-theory and triangulated categories},
  Invent. Math. \textbf{150} (2002), no.~1, 111--116. \MR{MR1930883
  (2003h:18015)}

\bibitem{Schlichting}
\bysame, \emph{Negative {$K$}-theory of derived categories}, Math. Z.
  \textbf{253} (2006), no.~1, 97--134. \MR{MR2206639 (2006i:19003)}

\bibitem{TTGrothFest}
R.~W. Thomason and Thomas Trobaugh, \emph{Higher algebraic {$K$}-theory of
  schemes and of derived categories}, The Grothendieck Festschrift, Vol.\ III,
  Progr. Math., vol.~88, Birkh\"auser Boston, Boston, MA, 1990, pp.~247--435.
  \MR{MR1106918 (92f:19001)}

\bibitem{ToenVezzosi}
Bertrand To{\"e}n and Gabriele Vezzosi, \emph{A remark on {$K$}-theory and
  {$S$}-categories}, Topology \textbf{43} (2004), no.~4, 765--791.
  \MR{MR2061207 (2005e:19001)}

\bibitem{Wald}
Friedhelm Waldhausen, \emph{Algebraic {$K$}-theory of spaces}, Algebraic and
  geometric topology (New Brunswick, N.J., 1983), Lecture Notes in Math., vol.
  1126, Springer, Berlin, 1985, pp.~318--419. \MR{MR802796 (86m:18011)}

\bibitem{WeissHammock}
Michael Weiss, \emph{Hammock localization in {W}aldhausen categories}, J. Pure
  Appl. Algebra \textbf{138} (1999), no.~2, 185--195. \MR{MR1689629
  (2000g:18014)}

\end{thebibliography}
\def\MR#1{}

\providecommand{\bysame}{\leavevmode\hbox to3em{\hrulefill}\thinspace}
\providecommand{\MR}{\relax\ifhmode\unskip\space\fi MR }
\providecommand{\MRhref}[2]{%
  \href{http://www.ams.org/mathscinet-getitem?mr=#1}{#2}
}
\providecommand{\href}[2]{#2}

\end{document}